\documentclass[oneside]{amsart}
  
\usepackage{amssymb}

\usepackage{graphicx}

\usepackage[cmtip,all]{xy}
\usepackage[latin1]{inputenc}
\usepackage{amssymb,amsmath}
\usepackage{verbatim}
\usepackage{color}
\usepackage{enumerate}
\usepackage{url}
\usepackage{ifthen}
\usepackage{geometry}
\usepackage[colorlinks]{hyperref}
\hypersetup{citecolor=blue}
\usepackage{todonotes}
\usepackage{ mathrsfs }

\usepackage{array}

\newtheorem{theorem}{Theorem}[section]
\newtheorem{lemma}[theorem]{Lemma}
\newtheorem{proposition}[theorem]{Proposition}
\newtheorem{conjecture}[theorem]{Conjecture}

\theoremstyle{definition}

\newtheorem{example}[theorem]{Example}

\theoremstyle{remark}
\newtheorem{remark}[theorem]{Remark}

\numberwithin{equation}{section}

\newcommand\Z{\ensuremath{\mathbb Z}}
\newcommand\Q{\ensuremath{\mathbb Q}}\newcommand\R{\ensuremath{\mathbb R}}
\newcommand\C{\ensuremath{\mathbb C}}\newcommand\F{\ensuremath{\mathbb F}}

\newcommand\coker{\operatorname{coker}}

\newcommand\disc{\operatorname{disc}}

\newcommand\Div{\operatorname{Div}}
\newcommand\End{\operatorname{End}}

\newcommand\Gal{\operatorname{Gal}}
\newcommand\GL{\operatorname{GL}}
\newcommand\Hom{\operatorname{Hom}}

\newcommand\Ind{\operatorname{Ind}}

\newcommand\M{\operatorname{M}}

\newcommand\Nm{\operatorname{Nm}}

\newcommand\PGL{\operatorname{PGL}}
\newcommand\PSL{\operatorname{PSL}}

\newcommand\sgn{\operatorname{sgn}}
\newcommand\SL{\operatorname{SL}}

\newcommand\HC{\operatorname{HC}}

\newcommand{\cH}{\mathcal{H}}
\newcommand{\cT}{\mathcal{T}}
\newcommand{\fp}{{\mathfrak{p}}}

\newcommand{\fm}{{\mathfrak{m}}}
\newcommand{\fn}{{\mathfrak{n}}}

\newcommand{\fq}{{\mathfrak{q}}}
\newcommand{\cU}{{\mathcal{U}}}
\newcommand{\cV}{{\mathcal{V}}}
\newcommand{\cE}{{\mathcal{E}}}
\newcommand{\cF}{{\mathcal{F}}}

\newcommand{\tto}[1]{%
\ifthenelse{\equal{#1}{}}{\to}{\stackrel{#1}{\to}}}

\def\cO{{\mathcal O}}

\newcommand{\smtx}[4]{\left(\begin{smallmatrix}#1&#2\\#3&#4\end{smallmatrix}\right)}

\newcommand{\tns}{\otimes}

\def\M{\operatorname{M}}

\def\P{\mathbb P}
\newcommand{\comp}{\begin{picture}(6,5)(-3,-2)\put(0,1){\circle{2}} \end{picture}}\def\circ{\comp}
\def\fN{\mathfrak N}

\def\ol{\overline}
\def\res{\operatorname{res}}

\newcommand{\ra}{\rightarrow}

\newcommand{\lra}{\longrightarrow}

\newcommand{\QQ}{\Q}

\newcommand{\ZZ}{\Z}
\newcommand{\fd}{\mathfrak d}
\DeclareMathOperator{\GG}{\mathbf{G}}

\def\Xint#1{\mathchoice
{\XXint\displaystyle\textstyle{#1}}%
{\XXint\textstyle\scriptstyle{#1}}%
{\XXint\scriptstyle\scriptscriptstyle{#1}}%
{\XXint\scriptscriptstyle\scriptscriptstyle{#1}}%
\!\int}
\def\XXint#1#2#3{{\setbox0=\hbox{$#1{#2#3}{\int}$}
\vcenter{\hbox{$#2#3$}}\kern-.5\wd0}}

\newcommand{\cHtwo}{\cH}
\newcommand{\cHthree}{\mathbb{H}}
\newcommand{\cerednikdrinfeld}{\v{C}erednik--Drinfel'd}
\begin{document}

\title{Darmon points on elliptic curves over number fields of arbitrary signature}

\author{Xavier Guitart}
\address{Institut f\"ur Experimentelle Mathematik\\
Universit\"at Duisburg--Essen\\
Germany}
\email{xevi.guitart@gmail.com}
\urladdr{https://www.uni-due.de/~hx0149/}
\thanks{The first author is supported by SFB/TR 45}

\author{Marc Masdeu}
\address{Mathematics Institute \\
University of Warwick \\
United Kingdom}
\email{M.Masdeu@warwick.ac.uk}
\urladdr{http://warwick.ac.uk/mmasdeu/}

\author{Mehmet Haluk \c{S}eng\"un}
\address{Mathematics Institute \\
University of Warwick \\
United Kingdom}
\email{M.H.Sengun@warwick.ac.uk}
\urladdr{http://warwick.ac.uk/haluksengun}
\thanks{The third author is supported by a Marie Curie Fellowship.}

\subjclass[2010]{11G40 (11F41, 11Y99)}

\date{\today}

\dedicatory{}

\begin{abstract}
We present new constructions of complex and $p$-adic Darmon points on elliptic curves over base fields of arbitrary signature. We conjecture that these points are global and present numerical 
evidence to support our conjecture. 
\end{abstract}
\maketitle 

\tableofcontents
\section{Introduction}

Let $F$ be a totally real number field and let $E/F$ be a \emph{modular} elliptic curve of conductor $\fN$. This means that there is a weight $2$ Hilbert modular form $f_E$ over $F$ which is a newform with integral Hecke eigenvalues and satisfies $L(f_E,s)=L(E/F,s)$. Let $K$ be \emph{any} quadratic extension of $F$. The modularity of $E$ implies the analytic continuation and functional equation of the completed $L$-function of the base change of $E$ to $K$:
\begin{align}\label{eq: functional eq} 
  \Lambda(E/K,s)=\operatorname{sign}(E/K)\Lambda(E/K,2-s),
\end{align}
where $\operatorname{sign}(E/K)\in\{\pm 1\}$.

It is well-known (see \cite[Section 3.6]{darmon-book}) that if $\textrm{sign}(E/K)=-1$ the celebrated Birch and Swinnerton-Dyer conjecture (BSD) predicts that there should be a systematic collection of non-torsion points on $E$ defined over suitable ring class fields of $K$. Indeed, if $H$ is a ring class field of $K$ of conductor coprime to $\fN$ the $L$-function of $E/H$ factors as
\begin{align*}
  L(E/H,s)= \prod_{\chi \colon {\Gal(H/K)}\ra\C^\times}L(E/K,\chi,s),
\end{align*}
and each twisted $L$-function $L(E/K,\chi,s)$ is known to satisfy a functional equation analogous to \eqref{eq: functional eq}, but with a sign that is independent of $\chi$ and which, in fact, equals $\operatorname{sign}(E/K)$. We obtain the following consequence of BSD.
\begin{conjecture}\label{NES}
If $\operatorname{sign}(E/K)=-1$ then \[\operatorname{rank}(E(H)) = [H\colon K]\] if and only if $L'(E/K,\chi,1)\neq 0$ for all characters $\chi$ of $\Gal(H/K)$. 
\end{conjecture}
If $K$ is totally imaginary and $E$ admits a parametrization from a Shimura curve, which happens when $E$ is modular and the Jacquet--Langlands 
hypothesis is satisfied (either the degree of $F$ is odd or there exists a prime dividing $\mathfrak{N}$ exactly), we know that Conjecture~\ref{NES} holds as a result of works of Gross--Zagier \cite{gross-zagier}, Kolyvagin \cite{kolyvagin}, Zhang \cite{zhang}, Bertolini--Darmon \cite{bertolini-darmon-kolyvagin}, and Tian \cite{tian}. Key to these works is the construction of a canonical collection of points on $E$, called \emph{Heegner points}, arising from the complex multiplication theory on Shimura curves.

No other cases of Conjecture~\ref{NES} have been proven besides this one. One of the main difficulties stems from the fact that, when $K$ is not totally imaginary, no modular method seems to be available for systematically manufacturing points over ring class fields of $K$. In spite of these obstructions, Darmon envisioned  in his influential paper~\cite{darmon-integration} a \emph{conjectural} $p$-adic analytic construction of global points on elliptic curves over $\Q$ which should be defined over ring class fields of certain \emph{real} quadratic fields. Over the years there have been many works which, building on the ideas of Darmon in~\cite{darmon-integration}, introduce conjectural constructions of global points on elliptic curves in settings that go beyond the classical one. Such points are indistinctly referred to in the literature either as \emph{Stark--Heegner points} or as~\emph{Darmon points}.

The common feature shared by all of these constructions is that the points are defined by \emph{local analytic} methods. Namely, they involve a certain place $v$ of $F$ (which can be either archimedean or non-archimedean) which does not split in $K$ and such that the ring class field $H$ is contained in the completion $K_v$, together with a recipe yielding a point $P_H\in E(K_v)$. The formula giving $P_H$ bears some resemblance with the local formulas for classical Heegner points (which will be reviewed in Section~\ref{subsection: HP}), and the  $P_H$ are conjectured to be global, belonging to $E(H)$, and to enjoy analogous properties to those satisfied by Heegner points.


This plethora of constructions can be divided into two types, according to the nature of the local field:
\begin{itemize}
\item \emph{Non-archimedean Darmon points}, in which $K_v$ is a $p$-adic field and the rigid analytic geometry of the $p$-adic upper half plane plays a distinguished role. This is the case of the original construction of Darmon \cite{darmon-integration}, as well the subsequent generalizations of Greenberg \cite{Gr}, Dasgupta \cite{Das}, and Longo--Rotger--Vigni \cite{LRV}.
\item \emph{Archimedean Darmon points}, where $K_v=\C$ and complex analytic methods are used. The first archimedean construction was also given by Darmon himself \cite[\S8]{darmon-book}, in the case where $K$ is almost totally real (ATR) and a certain Heegner-type condition is satisfied, and later generalized by Gartner \cite{Ga-art}.
\end{itemize}

These conjectures are supported by abundant numerical evidence \cite{darmon-green}, \cite{darmon-pollack}, \cite{guitart-masdeu-h}, \cite{shpquat}, \cite{darmon-logan}, \cite{guitart-masdeu}. Actually, a theme that runs in parallel with providing constructions of Darmon points is that of finding the algorithms that allow for their effective computation in order to test the conjectures.

Suppose now that $E$ is defined over a number field $F$ of signature $(r,s)$, i.e., with $r\geq 0$ real places and $s\geq 0$ complex places. In this more general framework the notion of modularity can be phrased as the existence of a cohomological modular form $f_E$ of weight $2$ for $\GL_2$ over $F$ such that $L(f_E,s)=L(E/F,s)$. The majority of modularity results concern the case where $F/\Q$ is totally real, and in this case the extensive work of Wiles and others (see~\cite{wiles, taylor-wiles, breuil} for $F=\Q$ and Skinner--Wiles~\cite{skinner-wiles-nearly-ordinary} for $[F\colon \Q]>0$) ensures the modularity of $E/F$ under some mild assumptions (recent work of Freitas, Le Hung and Siksek~\cite{FLS} shows that if $F/\Q$ is real quadratic then \emph{all} elliptic curves over $F$ are modular). In contrast, when $s>0$ there has been very little progress towards the analogous results, except for extensive numerical verifications (e.g. \cite{grunewald, cremona, whitley,gunnells_5, gunnells_23}).

It is natural to expect that Conjecture \ref{NES} holds for general $F$. If one assumes that $E/F$ is modular and $K/F$ is a quadratic extension, some standard conjectures concerning the functional equations of twisted $L$-functions (cf. \cite[Remark 3.18]{darmon-book}), allow for Conjecture~\ref{NES} to be regarded as a consequence of BSD again. So if $\operatorname{sign}(E/K)=-1$, one also expects the existence of an abundant supply of non-torsion points defined over ring class fields of $K$. In this general context, Trifkovi\'c introduced in~\cite{trifkovic} a non-archimedean construction of Darmon points that applies to curves defined over imaginary quadratic base fields. Therefore, up to now Darmon points have been constructed only for curves $E/F$ where either $F$ is \emph{totally real} or \emph{imaginary quadratic}.

The main contributions of the present article are new constructions of both \emph{archimedean} and \emph{non-archimedean} Darmon points on elliptic curves defined over \emph{arbitrary} number fields $F$. They work over any pair $(E/F,K/F)$, under the minimal assumptions that $E/F$ is modular and $\operatorname{sign}(E/K)=-1$, which are inherent to the method. We present our constructions under the additional assumption that the narrow class number of $F$ is one. This is done in order to make the method more transparent and concrete, and the general class number case should be treated adelically as in G\"artner \cite{Ga-art}. Our construction coincides with the previous ones in the cases where $F$ is totally real or imaginary quadratic, and thus this work can be regarded as an extension of all the previous constructions to arbitrary base fields.

In addition to the new constructions, we take advantage of the concrete nature of the method and carry out numerical experiments to test the validity of our conjectures. In particular, we provide the first numerical examples of Darmon points on elliptic curves over number fields with mixed signatures, and verify that our approximations are very close to global points.

For the convenience of the reader, and in order to describe more precisely our results, we devote the rest of the introduction to outline our main constructions (relegating the details to Sections \ref{section: archimedean darmon points} and \ref{section: non-archimedean Darmon points}). In fact, building on the elegant cohomological perspective introduced by Greenberg \cite{Gr}, we present them in a language that unifies, at least formally, the archimedean and the non-archimedean constructions. That is to say, although the techniques used in the two cases are rather different, namely $p$-adic and complex analytic, the recipe giving the points $P_H$ can be expressed in a formalism that encompasses the two types of points simultaneously. We hope that providing this unified framework may help in the future to clarify the relationship between the two kinds of points which at the moment, and to the best of our knowledge, seems rather mysterious.

We begin by recalling the classical construction of Heegner points, presented in a way that motivates the constructions of this paper.

\subsection{Revisiting classical Heegner points}\label{subsection: HP}
Let $E$ be an elliptic curve over $\Q$ of conductor $N$, which we assume squarefree for simplicity. Let $K$ be a quadratic imaginary field of discriminant coprime to $N$, and suppose $\text{sign}(E/K)=-1$. Let $\Sigma_\Q$ be the set of places of $\Q$, and denote by $\infty\in \Sigma_\Q$ the archimedean place. Define the following subset $S(E,K)$ of $\Sigma_\Q$:
\begin{align*}
  S(E,K)=\{p\mid N \colon p \text{ is inert in } K\}\cup \{\infty\}= \{ v\in \Sigma_\Q \colon v\mid N\infty \text{ and } v \text{ does not split in } K \}.
\end{align*}It is well-known that $\text{sign}(E/K)=(-1)^{\# S(E,K)}$, hence the negative sign assumption is equivalent to $S(E,K)$ having odd cardinality. 

The set $S(E,K)$ induces a factorization $N=DM$, where $D$ is simply the product of the (finite) primes in $S(E,K)$. Let $B/\Q$ be the (indefinite) quaternion algebra of discriminant $D$ and let $R\subset B$ be an Eichler order of level $M$. We denote by $X_0(D,M)/\Q$ the Shimura curve associated to this data, which is the coarse moduli space classifying pairs $(A,\iota)$ where $A$ is an abelian surface and $\iota$ is an inclusion $ R\hookrightarrow \End(A)$. Since $E$ is modular there exists a surjective morphism defined over $\Q$
\begin{align}\label{eq:pi}
 \pi\colon J_0(D,M)=\operatorname{Jac}(X_0(D,M))\lra E.
\end{align}
Let $\cO\subset K$ be an order of conductor relatively prime to $N$, and let $\operatorname{CM}(\cO)\subset X_0(D,N)$ denote the set of CM points attached to $\cO$. It is well-known that these points are defined over the ring class field $H_\cO$ of $\cO$. Thus any degree zero divisor $C\in\Div^0(\operatorname{CM}(\cO))$ gives rise to a point in $J_0(D,M)(H_\cO)$; its projection under $\pi$ is what is known as a Heegner point on $E$, and it is also defined over $H_\cO$.

Although the modular parametrization \eqref{eq:pi} is algebraic and defined over $\Q$, one can obtain useful analytic expressions for it by extending scalars to the completion $K_v$ of $K$ at a place $v\in S(E,K)$. The analytic uniformizations of $J_0(D,N)$ and $E$, and hence the resulting formulas for the Heegner points, come in two flavors, depending on whether the chosen place is $v=\infty$ or $v=p$.

\subsubsection{Heegner points via archimedean uniformization} This corresponds to the case $v=\infty$, so that $K_v=\C$. The analytic structure of $E(\C)$ is given by the Weierstrass uniformization $E(\C)\simeq \C/\Lambda_E$, for some lattice $\Lambda_E=\Z\oplus\Z\tau$. In order to unify the notation with the non-archimedean case, it is convenient to consider instead the Tate uniformization map $ \eta_{\text{Ta}} \colon \C^\times/q^\Z \ra E(\C), \ \text{ where } q= \exp(\tau)= e^{2\pi i \tau}.$

Let $B$ be as above, which we recall is the quaternion algebra ramified at the places $S(E,K)\setminus\{ v \}=S(E,K)\setminus\{\infty\}$. Let $\Gamma\subset B$ denote the group of norm $1$ elements of $R^\times$, which acts on the complex upper half plane $\cHtwo=\{z\in\C\colon \operatorname{Im}(z)>0\}$ by means of a fixed splitting $B\otimes_\Q \R \simeq \M_2(\R)$. The complex uniformization of $X_0(D,M)$ is then given by the isomorphism\footnote{If $B=\M_2(\Q)$ then $\Gamma\backslash \cH$ is actually isomorphic to the \emph{open} modular curve.} \begin{align}\label{eq:arch unif}X_0(D,M)(\C)\simeq \Gamma \backslash \cHtwo.\end{align}

Let $f_E\in S_2(\Gamma)$ be the weight two newform for $\Gamma$ attached to $E$ by means of the Modularity Theorem and the Jacquet--Langlands correspondence. That is to say, $f_E$ is a holomorphic function on $\cHtwo$ such that the differential $\omega_E=2\pi i f_E(z)dz$ is invariant under $\Gamma$, vanishes at the cusps if $B=\M_2(\Q)$, and $T_\ell \omega_E=a_\ell\omega_E$ for almost all primes $\ell$ (where $T_\ell$ denotes the $\ell$-th Hecke operator and $a_\ell=\ell+1-\#E(\F_\ell)$). Under the identification $\Div^0(\Gamma\backslash \cHtwo)\simeq J_0(D,M)(\C)$, the complex analytic incarnation of \eqref{eq:pi} is the map
\[
\xymatrix@R3pt{
\Div^0(\Gamma\backslash\cHtwo)\ar[r]&E(\C)\\
\tau_2-\tau_1\ar@{|->}[r]&\eta_{\text{Ta}}\left( \displaystyle\Xint\times_{\tau_1}^{\tau_2}\omega_E \right)
}
\]
where $\Xint\times_{\tau_2}^{\tau_1}\omega_E$ is a shorthand for $\exp(\int_{\tau_2}^{\tau_1}\omega_E)$.

Let now $\psi \colon \cO\hookrightarrow R$ be an optimal embedding and denote by $\tau_\psi\in \cHtwo$ the single fixed point of $\cO$ acting on $\cHtwo$ via $\psi$. The class $[\tau_\psi]$ of $\tau_\psi$ in $\Gamma\backslash \cHtwo$ corresponds to a CM point under the identification \eqref{eq:arch unif}, and in fact all elements of $\operatorname{CM}(\cO)$ arise in this way for some $\psi$. One associates to $[\tau_\psi]$ the following degree $0$ divisor: we consider $\tau_\psi^0 = (T_\ell-\ell-1)[\tau_\psi]\in \Div^0(\operatorname{CM}(\cO))$. In this way, to any $\psi$ we can associate the Heegner point
\begin{equation}\label{eq:P_psi}
  P_\psi = \eta_{\text{Ta}}\left(\Xint\times_{\tau_\psi^0}\omega_E  \right)\in E(H_\cO)\subset E(\C).
\end{equation}
We remark that there are several ways of associating a degree $0$ divisor to $[\tau_\psi]$, and that this is not the one that is usually considered in the literature; for instance, when $B=\M_2(\Q)$ the degree $0$ divisor that is usually associated to $\psi$ is  $[\tau_\psi]-[i\infty]$. However, the one presented above is more appropriate for later generalizing to Darmon points.

In order to motivate the structure of the Darmon point construction that will be given below, let us reinterpret the three ingredients that appear in \eqref{eq:P_psi} in a slightly different manner. First, the differential $\omega_E$ can be seen as a differential $1$-form in $\cHtwo$ that is invariant under the action of $\Gamma$; that is to say, we can view $\omega_E$ as an element in $H^0(\Gamma,\Omega^1_{\cHtwo})$. Second, the identification $\Div(\Gamma\backslash \cHtwo)\simeq H_0(\Gamma,\Div\cHtwo)$ allows us to regard $[\tau_\psi]$ as an element in $H_0(\Gamma,\Div\cHtwo)$. Observe that $(T_\ell-\ell-1)[\tau_\psi]$ lies in the image of the natural map
\begin{align*}
  H_0(\Gamma,\Div^0\cHtwo)\lra H_0(\Gamma,\Div\cHtwo),
\end{align*}
and we let $\Delta_\psi$ be a preimage of $(T_\ell-\ell-1)[\tau_\psi]$ in $ H_0(\Gamma,\Div^0\cHtwo)$. Finally, since the natural integration pairing $\Xint\times\colon \Omega^1_{\cHtwo}\times \Div^0\cHtwo\ra \C^\times$ (given by exponentiation of the line integral) is equivariant with respect to the action of $\GL_2(\R)$, it induces a well defined pairing 
\begin{align*}
  \Xint\times \colon H^0(\Gamma,\Omega^1_{\cHtwo})\times H_0(\Gamma,\Div^0\cHtwo)\ra \C^\times.
\end{align*}
Summing up, we see that the Heegner point $P_\psi$ of \eqref{eq:P_psi} is obtained by means of the following recipe:
\begin{enumerate}
\item the elliptic curve $E$ gives rise to a cohomology class $\omega_E\in H^0(\Gamma,\Omega^1_{\cHtwo})$;
\item the optimal embedding $\psi$ gives a homology class $\Delta_\psi\in H_0(\Gamma,\Div^0 \cHtwo)$; and
\item  $P_\psi$ is the image under Tate's uniformization of \begin{align}\label{eq: formula archimedean 1} J_\psi=\Xint\times_{\Delta_\psi}\omega_E\in \C^\times/q^\Z.\end{align}
\end{enumerate}
\subsubsection{Heegner points via non-archimedean uniformization} In this case the place $v\in S(E,K)$ at which we localize \eqref{eq:pi} is a prime $p$, so that $K_v=K_p$ is the quadratic unramified extension of $\Q_p$. Of  course, this is only possible if $S(E,K)\setminus \{\infty\}$ is non-empty. Since $p$ is a prime of multiplicative reduction the analytic description of $E(K_p)$ is given by the Tate uniformization $\eta_{\text{Ta}}\colon K_p^\times /q^\Z\ra E(K_p)$, where now $q\in \Q_p^\times$ denotes the Tate parameter of $E$.

The local description of $X_0(D,M)(K_p)$ is given by the \cerednikdrinfeld{} uniformization as follows. Let $\mathcal{B}/\Q$ denote the quaternion algebra ramified at the places $S(E,K)\setminus\{ v \}=S(E,K)\setminus \{p\}$. Choose $\mathcal{R}\subset \mathcal{B}$ a $\Z[1/p]$-Eichler order of level $M$, and let now $\Gamma$ denote the group of norm $1$ elements of $\mathcal{R}^\times$. The group $\Gamma$ acts on the $K_p$-points of the $p$-adic upper half plane, which we denote by $\cH_p=K_p\setminus\Q_p$, by means of a fixed splitting $\mathcal{B}\otimes_\Q\Q_p\simeq \M_2(\Q_p)$. The \cerednikdrinfeld{} uniformization provides with an isomorphism of rigid analytic spaces
\begin{align}\label{eq:non-arch unif}
  X_0(D,M)(K_p)\simeq \Gamma\backslash \cH_p.
\end{align}
Let $S_2(\Gamma)$ denote the space of weight two rigid analytic modular forms for $\Gamma$. It consists of rigid analytic functions on $\cH_p$ satisfying certain invariance properties with respect to the action of $\Gamma$. Thanks to the Modularity Theorem and the \cerednikdrinfeld{} uniformization, there exists a newform $f_E\in S_2(\Gamma)$ such that $T_\ell f_E=a_\ell f_E$ for almost all primes $\ell$. By results of Teitelbaum (see for example~\cite{teitelbaum}) $S_2(\Gamma)$ can be identified with the space of rigid analytic differential forms on $\Gamma\backslash\cH_p$. Equivalently,  $f_E$ can be identified with a rigid analytic differential on $\cH_p$ which is invariant under the action of $\Gamma$, and thus with a cohomology class $\omega_E\in H^0(\Gamma,\Omega^1_{\cH_p})$.

Now let $\psi\colon \cO \hookrightarrow \mathcal R$ be an optimal embedding of $\Z[1/p]$-algebras. Let $\tau_\psi\in\cH_p$ be one the two fixed points of $\cO$, acting on $\cH_p$ by means of $\psi$. The image $[\tau_\psi]$ of $\tau_\psi$ in $\Gamma\backslash \cH_p$ corresponds to an element of $\operatorname{CM}(\cO)$ under \eqref{eq:non-arch unif}, and all the points in $\operatorname{CM}(\cO)$ arise in this manner for some $\psi$. Similarly as before we interpret $[\tau_\psi]$ as lying in $H_0(\Gamma,\Div\cH_p)$, and we let $\Delta_\psi $ be a preimage of $ (T_\ell-\ell - 1)[\tau_\psi]$ under the natural map $H_0(\Gamma,\Div^0\cH_p)\ra H_0(\Gamma,\Div\cH_p)$.


By results of Bertolini--Darmon \cite{bertolini-darmon-p-adic-periods} the $p$-adic formula for \eqref{eq:pi} is given explicitly in terms of the so-called \emph{multiplicative line integrals} of $f_E$ (see \ref{subsection: non-archimedean integration} below for the precise definition). Namely, under the identification $\Div^0(\Gamma\backslash \cH_p)\simeq J_0(D,M)(K_p)$ it is of the form
\[
\xymatrix@R5pt{
 \Div^0(\Gamma\backslash\cH_p) \ar[r] & E(K_p)\\
\tau_2-\tau_1\ar@{|->}[r]& \eta_{\text{Ta}}
\left( \begin{displaystyle} \Xint\times_{\tau_1}^{\tau_2} \end{displaystyle} f_E \right).
}
\]

The multiplicative line integrals give rise to a well defined pairing at the level of (co)homology
\begin{align*}
  \Xint\times \colon H^0\left(\Gamma,\Omega^1_{\cH_p}\right)\times H_0\left(\Gamma,\Div^0(\cH_p)\right)\lra K_p^\times,
\end{align*}
and we can attach a Heegner point to $\psi$ by means of the formula
\begin{align}\label{eq:P_psi p-adic}
  P_\psi = \eta_{\text{Ta}}\left(\Xint\times_{\Delta_\psi}\omega_E  \right)\in E(H_\cO)\subset E(K_p).
\end{align}
If we compare with the expression for the Heegner point obtained in the archimedean case we see that, although the ingredients that appear in it are fundamentally of a different nature, the formal structure underlying the construction is rather similar. Namely:
\begin{enumerate}
\item the elliptic curve $E$ gives rise to a cohomology class $\omega_E\in H^0(\Gamma,\Omega^1_{\cH_p})$;
\item the optimal embedding $\psi$ gives a homology class $\Delta_\psi\in H_0(\Gamma,\Div^0 \cH_p)$; and
\item  $P_\psi$ is the image under Tate's uniformization of
\begin{align}\label{eq: formula archimedean 2} \Xint\times_{\Delta_\psi}\omega_E\in \C_p^\times/q^\Z.\end{align}
\end{enumerate}

Suppose now that $E$ is a modular elliptic curve defined over an arbitrary number field $F$, and that $K/F$ is a quadratic extension with $\operatorname{sign}(E/K)=-1$. If $K/F$ is a CM extension and the Jacquet--Langlands condition holds, then these constructions generalize: there is a modular parametrization generalizing \eqref{eq:pi}, and the Heegner point construction has been generalized by Zhang \cite{zhang}.

However, if $K/F$ is not CM then one is at a loss of generalizing the above strategy. In fact, if $F$ is not totally real then no analogue of the modular parametrization \eqref{eq:pi} is known (or even expected to exist). The general philosophy underlying what has been known as Darmon points can be viewed as seeking appropriate analogues of the analytic formulas \eqref{eq: formula archimedean 1} and \eqref{eq: formula archimedean 2}, even (and specially) in those situations where no natural generalization of the geometric counterpart \eqref{eq:pi} is available. In the next subsection we give an overview of the main construction of this paper which, following ideas of Greenberg \cite{Gr}, replaces the $H_0$ and $H^0$ above by higher (co)homology groups.

\newcommand{\HH}{\mathbb{H}}
\subsection{Darmon points on curves over arbitrary number fields} \label{construction-main}

The following notations will be in force for the rest of the article. Let $F$ be a number field of narrow class number one (in order to avoid the adelic language) with signature $(r,s)$, i.e. $r\geq 0$ real places and $s\geq 0$ complex places. Let $E/F$ be a modular elliptic curve of conductor $\fN$ which, for simplicity, we assume to be squarefree. Let $K/F$ be a quadratic extension whose discriminant is coprime to $\fN$. Let $\infty_F=\{\sigma_1,\dots,\sigma_{r+s}\}$ denote the set of infinite places of $F$ and let $0\leq n\leq r$ be the number of real places of $F$ which split in $K$. The set $\infty_F$ can be partitioned into three (each of them possibly empty) sets:
\begin{itemize}
\item $\{\sigma_1,\ldots, \sigma_{n}\}\text{ are real places which split in $K$ (i.e. $K\tns_{F,\sigma_i} \R = \R\oplus\R$)}$,
\item $\{\sigma_{n+1},\ldots \sigma_r\} \text{ are real places which ramify in $K$ (i.e. $K\tns_{F,\sigma_i} \R = \C$)}$,
\item $\{\sigma_{r+1},\ldots \sigma_s\}\text{ are the complex places, which always split in $K$ (i.e. $K\tns_{F,\sigma_i} \C = \C\oplus\C$)}$.
\end{itemize}

Let us also define
\[
S(E,K) = \{ w \mid \fN\infty_F \colon w \text{ is non-split in } K\};
\]
that is to say, the set of places that are either archimedean or dividing the conductor and such that they do not split in $K$. The sign of the completed $L$-function of $E/K$ is conjectured (see \cite[Remark 3.18]{darmon-book}) to be given by
\[
\operatorname{sign} (E/K)=(-1)^{ \# S(E,K)}.
\]
From now on we assume this conjecture (which is known if $F$ is totally real). In addition, we suppose that the functional equation of $E/K$ has sign $-1$ or, equivalently, that $S(E,K)$ has odd cardinality.

Fix an element $v \in S(E,K)$. Let $B$ be the quaternion algebra over $F$ whose set of ramification places is given by 
\[
 \operatorname{ram}(B)=S(E,K) \setminus \{ v \}.
\]
Observe that this is possible precisely because of our assumption that $\#S(E,K)$ is odd. 

Since the place $v$ is not split in $K$, it extends to a unique place in $K$, which will also be denoted $v$. We write $F_v$ and $K_v$ for the completions of $F$ and $K$ at $v$. If $v$ is an archimedean place then $F_v=\R$ and $K_v=\C$. If $v$ is non-archimedean, say $v=\fp$, then $F_v$ is a finite extension of $\Q_p$ (where $p=\fp\cap \Z$) and $K_v$ is the quadratic unramified extension of $F_v$. 

The points that we define in this note are manufactured as local points, in the sense that by construction they lie in $E(K_v)$. The actual details of the construction vary depending on whether the chosen place $v$ is archimedean or non-archimedean. Nonetheless, in a similar way to what has been for Heegner points above, we present the construction in a way that formally unifies the archimedean and non archimedean cases.

Let $\fd$ be the discriminant of $B$, and let $\fm$ be the ideal of $F$ defined by the equation
\begin{align*}
\fN = \begin{cases} \fm\fd &\mbox{ if $v$ is archimedean} \\
\fp\fm\fd & \mbox{if } v=\fp. \end{cases}  
\end{align*}
Denote by $\cO_{F,\{v\}}$ the ring of $\{v\}$-integers in $F$ (this is just $\cO_F$ if $v$ is archimedean and the usual ring of $\{\fp\}$-integers $\cO_{F,\{\fp\}}$ if $v=\fp$). Let $R_0(\fm)$ be an Eichler order of level $\fm$ in $B$ and denote by $R_0(\fm)^\times_1$ its group of elements with reduced norm $1$. Let $\Gamma_0(\fm)\subset B^\times/F^\times$ be the image of $R_1^\times$ under the natural map $B^\times\ra B^\times/F^\times$. For simplicity let us assume that $\Gamma_0(\fm)$ does not have any elliptic element of finite order. Set $R=R_0(\fm)\otimes_{\cO_F}\cO_{F,\{v\}}$ and let $\Gamma\subset B^\times /F^\times$ denote the image of $R_1^\times$ in $B^\times/F^\times$.

Let $\cH=\{z\in\C\colon \operatorname{Im}(z)>0\}$ be the complex upper half plane. For a prime $\fp\in S(E,K)$ we put $\cH_\fp=K_\fp\setminus F_\fp$, the $K_\fp$-rational points of the $\fp$-adic upper half plane. It is notationally convenient to introduce:
\begin{align*}
\cH_v = \begin{cases} \cHtwo &\mbox{if } v \mbox{ is archimedean}  \\
\cH_\fp & \mbox{if } v=\fp . \end{cases}  
\end{align*}
We denote by $\cHthree$ the upper half space 
\begin{align*}
  \cHthree =\{(x,y)\in \C\times\R \colon y>0\}.
\end{align*}
Since $v$ does not ramify in $B$, we can fix an isomorphism 
\begin{align}\label{eq: splitting at v}\iota_v: B\otimes_v F_v \stackrel{\simeq}{\lra} M_2(F_v).\end{align} Fix also splitting  isomorphisms for the following archimedean places: 
\begin{align}\label{eq: splitting at real}
 \iota_i\colon B\otimes_{\sigma_i} \R \simeq \M_2(\R)\text{ for } i=1,\dots,n, \ \text{ and }
\end{align}
\begin{align}\label{eq: splitting at complex}
 \iota_i\colon B\otimes_{\sigma_i} \C \simeq \M_2(\C)\text{ for } i=r+1,\dots,s.
\end{align}
Observe that $B^\times$ acts on the space (see Section \ref{sec:review})
\begin{align*}
  \cH_v\times \cHtwo^{n}\times \cHthree^{s}
\end{align*}
by means of $\iota_{v}\times \iota_1\times\cdots \times \iota_{n}\times\iota_{r+1}\times\cdots \times \iota_{r+s}$. We stress that in the above product the first copy $\cH_v$ can be either the complex or the $\fp$-adic upper half plane, depending on whether $v$ is archimedean or non-archimedean.

Let $\Omega^1_{\cH_v}$ denote the space of regular differential $1$-forms in $\cH_v$. This means the usual holomorphic $1$-forms in $\cH$ if $v$ is archimedean, and the rigid analytic differential forms on $\cH_\fp$ if $v=\fp$. There is an integration pairing 
\begin{align*}
\begin{array}{cccc}
 \int\colon  & \Omega_{\cH_v}^1 \times \Div^0(\cH_v) &\lra& K_v\\
 & (\omega,\tau_2-\tau_1 ) & \longmapsto & \begin{displaystyle} \int_{\tau_1}^{\tau_2}\omega, \end{displaystyle}
\end{array}
\end{align*}
given by the usual complex line integral if $v$ is archimedean and by the Coleman integral if $v$ is non-archimedean. In order to define Darmon points, it is better (and in the non-archimedean case it is compulsory) to work with a multiplicative version of the above integration pairing. It turns out that if $v$ is non-archimedean, one can define multiplicative Coleman integrals only for differentials that have integer residues. Denote by $\Omega_{\cH_v}^1(\Z)$ the module of rigid analytic differentials with integer residues, which is the one that we will work with in the non-archimedean case. In order to lighten the notation, we set
\begin{align*}
  \Omega_v = \begin{cases} \Omega_{\cH_v}^1 &\mbox{if } v \mbox{ is archimedean},  \\
\Omega_{\cH_v}^1(\Z) & \mbox{if $v$ is non-archimedean} . \end{cases}  
\end{align*}
Then there is a well defined pairing (see \S\ref{subsection: non-archimedean integration} below the precise definition)
\begin{align}\label{eq:integration}
\begin{array}{cccc}
\begin{displaystyle}  \Xint\times\colon \end{displaystyle}  & \Omega_{v} \times \Div^0(\cH_v) &\lra& K_v^\times.
\end{array}
\end{align}

Let us fix a choice of ``signs at infinity'', which can be viewed as a map
\begin{align}
  \lambda_\infty \colon \{\sigma_1,\dots,\sigma_n\}\lra \{\pm 1\}.
\end{align}
The curve $E$, together with the above choice of signs, determines a character $\lambda_E$ of the Hecke algebra ``away from $\mathfrak{d}v$'' as follows:
\begin{align}\label{eq: lambda_E}
  \lambda_E([T_\mathfrak{l}])=a_\mathfrak{l}(E) \text{ for  primes } \mathfrak{l}\nmid \fd,  \  \text{ and } \ \lambda_E([T_{\sigma_i}])=\lambda_\infty(\sigma_i) \text{ for } i=1,\dots,n,
\end{align}
where $T_\mathfrak{l}$  (resp. $T_{\sigma_i}$) denotes the Hecke operator at the prime ideal $\mathfrak{l}$ (resp. at the infinite place $\sigma_i$), and $a_\mathfrak{l}(E)=|\mathfrak{l}|+1 - \#E(\cO_F/\mathfrak{l})$ (cf. \S\ref{subsection: operators} below for the definitions). 

Let $\cO$ be an order in $K$ such that $(\disc(\cO/\cO_F),\fN)=1$. Let $\operatorname{Pic}^+\cO$ denote the narrow Picard group of $\cO$, which class field theory identifies, via the reciprocity map, with the Galois group of the narrow ring class field $H_\cO^+$ of $\cO$:
\[\operatorname{rec}\colon  \operatorname{Pic}^+(\cO)\simeq \Gal(H_\cO^+/K).\]

An embedding of $\cO_F$-algebras $\psi\colon \cO \hookrightarrow R$ is said to be \emph{optimal} if it does not extend to an embedding of a larger order into $R$. Denote by $\cE(\cO,R)$ the set of optimal embeddings, which admits a well-known natural action of $\operatorname{Pic}^+(\cO)$.

 As we explain in the body 
of the paper, we construct: 
\begin{enumerate}
 \item A cohomology class $\Phi_E \in H^{n+s}(\Gamma, \Omega_v)$ attached to $\lambda_E$;
\item  A homology class $\Delta_\psi \in H_{n+s}(\Gamma,\Div^0\cH_v)$ associated to any $\psi\in\cE(\cO,R)$. In fact, $\Delta_\psi$ is only well defined modulo the subgroup $\delta (H_{n+s+1}(\Gamma,\Z))$, where $\delta$ denotes a certain connecting homomorphism. 
\item The integration pairing \eqref{eq:integration} induces, by cap product, a pairing at the level of (co)homology groups:
\begin{align*}
\begin{array}{cccc}
\begin{displaystyle}  \Xint\times\colon \end{displaystyle} & H^{n+s}(\Gamma,\Omega_{v}) \times H_{n+s}(\Gamma,\Div^0(\cH_v)) &\lra & K_v^\times.
\end{array}
\end{align*}
By letting $
  L = \left \{ \begin{displaystyle} \Xint\times_{\delta(c)} \end{displaystyle} \Phi_E\colon c \in H_{n+s+1}(\Gamma,\Z) \right \} \subset K_v^\times
$
we define
\begin{align*}
  J_\psi = \Xint\times_{\Delta_\psi}\Phi_E\in K_v^\times/L.
\end{align*}
\end{enumerate}
The following conjecture is an extension to base fields $F$ of mixed signature of conjectures of Oda \cite{oda} and Yoshida \cite{yoshida} if $v$ is archimedean, and of theorems of Darmon \cite{Da1}, Dasgupta--Greenberg \cite{greenberg-dasgupta}, and Longo--Rotger--Vigni \cite{LRV} if $v$ is non-archimedean and $F=\Q$.
\begin{conjecture}
  There exists an isogeny $\beta\colon K_v^\times/L \ra E(K_v)$.
\end{conjecture}
Granting this conjecture, we define the point 
\begin{align*}
  P_\psi = \beta(J_\psi)\in E(K_v).
\end{align*}
We regard $H_\cO^+$ as a subfield of $K_v$ by means of a fixed embedding. The following conjecture is an extension to base fields $F$ of mixed characteristic of the archimedean conjectures of Darmon \cite{darmon-book}, G\"artner \cite{Ga-art}, and of the non-archimedean conjectures of Darmon \cite{Da1}, Greenberg \cite{Gr}, and Trifkovi\'c \cite{trifkovic}.  
\begin{conjecture}
  The isogeny $\beta$ can be chosen in such a way that $P_\psi$ is the image of a global point in $E(H_\cO^+)$. In addition, for any $\alpha\in\operatorname{Pic}^+(\cO)$ we have that $P_{\alpha\cdot\psi}=\operatorname{rec}(\alpha)(P_\psi)$.
\end{conjecture}
In the rest of the paper we give the details of the construction outlined above: we make explicit the definition of the cohomology class $\Phi_E$, the homology class $\Delta_\psi$, and the integration pairing leading to $J_\psi$. As mentioned before, the actual construction is different, depending on the nature of the chosen place $v$: Section \ref{section: archimedean darmon points} is devoted to the case where $v$ is a archimedean, and Section \ref{section: non-archimedean Darmon points} to non-archimedean $v$. 

In Section \ref{section: effective methods} we present some effective algorithms for computing the points $P_\psi$ (both archimedean and non-archimedean) in the only setting that seems to be computationally treatable at the moment, which is when $n+s\leq 1$. They have been computed using our own Sage~\cite{sage} implementations\footnote{To be found at \url{https://github.com/mmasdeu/darmonpoints} and \url{https://github.com/mmasdeu/cubicpoints}.}, which in turn use some Magma~\cite{Magma} code of Aurel Page to work with Kleinian groups, as well as Pari~\cite{pari} routines for elliptic curves.  In particular, we are able to perform explicit computations of three new types of 
Darmon points, that are a novelty of the present paper:
\begin{enumerate}
\item In Example \ref{example: explicit computations and examples-arch}, we compute an archimedean point on a curve $E$ defined over a cubic field $F$ of signature $(1,1)$, over a field $K$ which is totally complex. In this case the algebra $B$ is the split matrix algebra.
\item In Example \ref{ex: non-arch 1}, we compute a non-archimedean point on a curve $E/F$, where $F$ is cubic of signature $(1,1)$ and $K$ is totally complex. In this case $B$ is a quaternion division algebra.
\item In Example \ref{ex: non-arch 2}, we compute non-archimedean points on a curve $E/F$ where $F$ is a quadratic imaginary field. A similar type of points had already been computed by Trifkovi\'c \cite{trifkovic} in the case where $B$ is the split matrix algebra. The novelty of our method and computations is that the underlying algebra $B$ is a division algebra.
\end{enumerate}
In all these examples we find that the computed points coincide with global points up to high numerical precision. This gives empirical evidence in support of the rationality of the Darmon points over mixed signature base fields constructed in this article.

\subsection{Acknowledgments} We thank the hopitality of the Hausdorff Research Institute for Mathematics (Bonn) and the Centro de Ciencias de Benasque Pedro Pascual (Benasque).
\section{Review of Cohomology and Modular Forms } \label{sec:review}

In this section, we review some fundamental definitions and facts about the cohomology $S$-arithmetic groups and modular forms for $\GL_2$ over number fields as they will be used in our constructions. 

Let us start by recalling some basics. Let $\cHtwo = \{ z \in \C : \Im(z)>0 \}$ denote the \emph{upper half-plane} model of the real hyperbolic 2-space. The Lie group $\PGL_2(\R)$ acts on $\cHtwo$ via the following rule
$$\begin{pmatrix} a & b \\  c & d \end{pmatrix} \cdot z = \begin{cases} \dfrac{az+b}{cz+d}, \ \ {\rm if} \ \ ad-bc > 0, \\  \\ \ \dfrac{a\bar{z}+b}{c\bar{z}+d}, 
\ \ {\rm if} \ \ ad-bc < 0. \end{cases}$$
Let $\cHthree = \{ z=(x,y) \in \C \times \R : y >0 \}$ denote the \emph{upper half-space} model of the real hyperbolic $3$-space. 
The Lie group $\PGL_2(\C)$ acts on $\cHthree$ via the rule
$$\begin{pmatrix} a & b \\  c & d \end{pmatrix} \cdot z = 
\left ( \dfrac{(ax+b)\overline{(cx+d)}+a\bar{c}y^2}
{|cx+d|^2+|cy|^2}, \dfrac{|ad-bc|y}{|cx+d|^2+|cy|^2} \right ).$$
Finally, given any number field $F$ and a prime ideal $\fp$ of $F$, consider the \emph{$\fp$-adic upper half-plane}
$\cHtwo_\fp = \P^1(K_\fp) \setminus \P^1(F_\fp)$, where $K_\fp$ is the 
unramified quadratic extension of $F_\fp$. The $p$-adic Lie group $\PGL_2(F_\fp)$ acts on $\cHtwo_\fp$ via 
$$\begin{pmatrix} a & b \\  c & d \end{pmatrix} \cdot z = \dfrac{az+b}{cz+d}.$$

\subsection{Operators on the cohomology}\label{subsection: operators}
As in Section \ref{construction-main}, let $F$ be a number field with narrow class number one, ring of integers $\cO_F$ and signature $(r,s)$. 
Let $B$ be a quaternion algebra over $F$  with discriminant $\fd$ that is split at $m \leq r$ real places $\{ \sigma_1, \hdots ,\sigma_m \}$ of $F$. Let us fix a finite set $S$ of places $F$ at which $B$ splits, including all such infinite ones. Let $\GG=\GG_B$ be the algebraic group over $F$ such that  $\GG(A) = (A \otimes_F B)^\times / A^\times $ for any $F$-algebra $A$, and let $\Gamma$ be an $S$-arithmetic subgroup of $\GG(F)$. View $\Gamma$ as a discrete subgroup of the group 
$$G=\prod_{v \in S} \GG(F_v) = \PGL_2(\R)^m \times \PGL_2(\C)^s \times \prod_{\substack{v \in S \\ v \ {\rm finite}}} \PGL_2(F_v).$$ 
Let $C(\Gamma)$ be the commensurator of $\Gamma$, that is, the group of all $g \in G$ such that the intersection of $\Gamma$ and $g\Gamma g^{-1}$ has finite index in both groups. 
Given a $\Z[C(\Gamma)]$-module $M$ and $g \in C(\Gamma)$, we form the linear map $T_g$ (called a \emph{Hecke operator}) acting on $H^i(\Gamma, M)$  
(action on homology is defined similarly) via the following diagram
$$\xymatrix{ H^i(\Gamma,M) \ar[d]^{res} & H^i(\Gamma,M) \\ H^i(\Gamma \cap g\Gamma g^{-1} ,M) \ar[r]^{\alpha} & H^i(g^{-1}\Gamma g \cap \Gamma ,M) \ar[u]^{cores} & }$$ 
where the map $\alpha$ is induced by conjugation by $g^{-1}$. 

In this paper, we will {\em only} consider $S$-arithmetic groups that arise via the following construction. Choose an Eichler $\cO_F$-order $R_0(\fn) \subset B$ of some level $\fn$ that is coprime to $\fd$ (see Section \ref{section: non-archimedean Darmon points}). Let $\Gamma=\Gamma_0(\fn)$ be the $S$-arithmetic group given by the image in $\GG(F)$ of the group of {\em reduced norm $1$} elements $R^\times_1$ where $R = R_0(\fn) \otimes_{\cO_F} \cO_{F,S}$ and $\cO_{F,S}$ is the ring of $S$-integers of $F$. 

Thanks to our global assumption that $F$ has narrow class number one, every ideal has a totally positive generator. Given a prime ideal $\mathfrak{l}$ of $F$ that is away from $S$ and coprime to $\fd$, choose such a generator $\pi$. Then there is an element $\gamma_\mathfrak{l} \in R^\times$ with norm $\pi$ that is (viewed as an element of $G$) in $C(\Gamma)$ (when $B$ is the matrix algebra, we can take 
$\gamma_\mathfrak{l}$ to be $\left ( \begin{smallmatrix} \pi & 0 \\ 0 & 1 \end{smallmatrix} \right )$). We will denote the associated Hecke operator with $T_\mathfrak{l}$ or $U_\mathfrak{l}$ depending on whether $\mathfrak{l} \mid \fn$ or not, respectively. 

We will also need the \emph{involutions at infinity} that act on the cohomology of $\Gamma$. Since $F$ has narrow class number one, there are units $\epsilon_1, \hdots, \epsilon_n$ in $\cO_F$ such that for $1 \leq i,j \leq n$, one has $\sigma_j(\epsilon_i) >0$ if and only if $i\not= j$. It follows from the Norm Theorem for quaternion algebras, see 
\cite[Section 2]{Gr} for details, that there are elements $\gamma_{\sigma_i} \in R^\times$ of norm $\epsilon_i$ that are (when viewed in $G$) in the normalizer of $\Gamma$ in $G$ (and thus in $C(\Gamma)$) and the Hecke operators $T_{\sigma_i}$ that they give rise to act as involutions on the cohomology of $\Gamma$. For example, if $B$ is the matrix algebra we can take $\gamma_{\sigma_i}$ to be $\left ( \begin{smallmatrix} \epsilon_i & 0 \\ 0 & 1 \end{smallmatrix} \right )$.

\subsection{Decomposition of the cohomology} 
\label{sec:decomposition-cohomology}
In this section we summarize some well-known results on the structure of the cohomology of arithmetic groups, mostly due to Harder, following \cite{harder-75, harder-87, li-schwermer}).

Let $\Gamma$ be an $S$-arithmetic group as above. Then as a discrete subgroup of $G$, it acts on the space 
$$X := \cHtwo^m \times \cHthree^s \times \prod_{\substack{v \in S \\ v \ {\rm finite}}} \cHtwo_v$$
properly and discontinuously. Since $\Gamma$ arises from reduced norm one elements of $R_0(\fn)$, the action is an \emph{orientation-preserving} isometry on every component. 

Consider the quotient space $X_\Gamma = \Gamma \backslash X$. 
As the universal cover $X$ is contractible, one can show that the cohomology of the space $X_\Gamma$ agrees with that of its fundamental group $\Gamma$, that is,
$$H^i(X_\Gamma, \C) \simeq H^i(\Gamma, \C).$$

When $\Gamma$ is arithmetic (that is, $S$ only consists of infinite places), $X_\Gamma = \Gamma \backslash X$ is a finite-volume Riemannian orbifold of dimension $2m+3s$ and one can say much more about the sructure of the cohomology of $X_\Gamma$. So for the rest of this section, let us assume that $\Gamma$ is arithmetic. Note that typically one assumes $X_\Gamma$ to be a smooth manifold for the results that will follow, but one can extend these to arbitrary $X_\Gamma$ by passing to a finite normal subgroup $\Gamma'\subset \Gamma$ and then taking $\Gamma/\Gamma'$-invariants.

Any cohomology class in $H^i(X_\Gamma,\C)$ can be represented by a harmonic differential $i$-form on $X_\Gamma$.  Moreover, there is Hecke module decomposition
$$H^i(X_\Gamma, \C) \simeq H^i_{univ}(X_\Gamma,\C) \oplus H^i_{cusp}(X_\Gamma,\C) \oplus H^i_{inf}(X_\Gamma,\C)$$
where the summands on the right hand side are called the \emph{universal subspace}, 
the \emph{cuspidal subspace} and the \emph{infinity subspace} respectively. The summand $H^i_{inf}(X_\Gamma,\C)$ is related to the cusps of $X_\Gamma$ 
so it only appears in the cases where $X_\Gamma$ is non-compact. 

Let $C^i(X_\Gamma)$ denote the space of cuspidal harmonic differential $i$-forms on $X_\Gamma$. Then there is an 
injection $C^i(X_\Gamma) \rightarrow H^i(X_\Gamma,\C)$ whose image defines the subspace $H^i_{cusp}(X_\Gamma,\C)$. 
It is well-known that (see, for example, \cite[Prop. 2.14.]{li-schwermer}) the cuspidal part of the cohomology vanishes outside a certain interval around the middle degree, 
namely 
\begin{equation}\label{eq: middle degree} H^i_{cusp}(\Gamma, \C) = \{ 0 \} \ \  \ \text{unless} \ \ \ m+s \leq i \leq  m+2s.\end{equation}

Let $A^i(X)$ denote the space of $G(\R)$-invariant differential $i$-forms on $X$. Such forms are closed 
and harmonic, and thus after descending to $X_\Gamma$ they give rise to cohomology classes:
$$A^i(X) \rightarrow H^i(X_\Gamma, \C).$$
The subspace $H^i_{univ}(X_\Gamma,\C)$ is the image of this map. Note that the $A^i(X)$ has a basis 
made of exterior products of combinations of the volume forms on the individual factors of $X$. 

Finally, in order to explain $H^i_{inf}(X_\Gamma,\C)$, assume that $X_\Gamma$ is non-compact. 
It is well-known that the Borel-Serre compactification $X_\Gamma^{BS}$ of $X$
is a compact manifold with boundary that is homotopy equivalent to $X_\Gamma$. In particular their cohomology groups agree. 
Consider the map on the cohomology groups induced by the restriction to the boundary map: 
$$\rho_i : H^i(X_\Gamma, \C) \simeq H^i(X_\Gamma^{BS}, \C) \rightarrow H^i(\partial X_\Gamma^{BS}, \C).$$
The subspace $H^i_{inf}(X_\Gamma,\C)$ is constructed by Harder as a section of the restriction map $\rho_i$ by use of harmonic Eisenstein series. Harder also showed that for $0 < i < m+s$, the map $\rho_i$ is trivial and thus $H^i_{inf}$ vanishes.

From the perspective of Hecke action, the interesting part of the cohomology is the cuspidal part. Indeed, it is well-known that if $c$ is a simultaneous eigenclass for the action of Hecke operators belonging to the universal part or the infinity part of the cohomology, then one has  $T_\mathfrak{l}\cdot c = (|\mathfrak{l}|+1) c$ for every prime ideal $\mathfrak{l}$, where $|\mathfrak{l}|$ denotes the norm of $\mathfrak{l}$.

Let us also discuss the \emph{new} subspaces of cohomology. Let $\fp$ be a prime divisor of $\fn$. Then there are two different maps from the cohomology of $\Gamma$ to that of $\Gamma_0(\fn /\fp)$. The intersection of the kernels of these two maps is denoted $H^i(\Gamma,\C)^{\fp {\text -new}}$ and is preserved by the Hecke action. For a square-free $\fd \mid \fn$, the $\fd$-new subspace $H^i(\Gamma,\C)^{\fd {\text -new}}$ 
is simply the intersection of all the $\fp$-new parts for $\fp \mid \fd$.

Finally, we would like to record a result of Blasius, Franke and Grunewald. Let $\Gamma_S$ and $\Gamma$ be $S$-arithmetic and arithmetic subgroups of $\GG(F)$ arising from Eichler orders of the same level $\fn$, as above. Observe that $\Gamma \subset \Gamma_S$. The following result is~\cite[Theorem 4]{grun-94}.
\begin{theorem} \label{blasius} The image of the restriction map $H^i(\Gamma_S, \C) \rightarrow H^i(\Gamma, \C)$ is precisely 
$H^i_{univ}(\Gamma,\C)$.
\end{theorem}

\subsection{Cohomological modular forms}
In this section we talk about cohomological modular forms following mainly Hida \cite{hida93, hida94}. Let $\Gamma$ be an arithmetic group as above. 
For $\gamma= \left ( \begin{smallmatrix} a & b  \\ c & d \end{smallmatrix} \right )$ in $\PGL_2(\R)$ and $z \in \cHtwo$, put 
$$j(\gamma, z) = cz+d,$$
for $\gamma= \left ( \begin{smallmatrix} a & b  \\ c & d \end{smallmatrix} \right )$ in $\PSL_2(\C)$ and 
$z=(x,y) \in \cHthree$, put
$$j(\gamma, z) = \begin{pmatrix} cx + d & -c y \\ \overline{c y} & \overline{c x + d} \end{pmatrix}.$$

Let $\Sigma_\R$ denote the real places of $F$ that split $B$ and fix some $J \subset \Sigma_\R$. 
A (column) vector-valued function 
$$f: X \longrightarrow \left ( \C^3 \right )^{\otimes s}$$ 
is called a weight $2$ type $J$ \emph{cuspidal modular form} for $\Gamma$ if   
\begin{equation}
\label{eq:cuspidal-modular-form}
f( \underline{\gamma} \cdot  \underline{z} ) = \Bigg ( \prod_{\sigma_i \in J} j(\gamma_i,z_i)^{-2} \prod_{\sigma_i \in \Sigma \setminus J} j(\gamma_i,\bar{z}_i)^{-2} 
\bigotimes_{i=m+1}^{m+s} {\text Sym^2}\left ( j(\gamma_i, z_i)^{-1}\right ) \Bigg )
\cdot f ( \underline{z} )
\end{equation}
for all $\underline{\gamma}=(\gamma_i)_{1\leq i \leq m+s } \in \Gamma$ and $ \underline{z} = (z_i)_{1\leq i \leq m+s } \in X$, and if 
$f$ is holomorphic, anti-holomorphic and harmonic with respect to each coordinate $z_i$ with $\sigma_i \in J$, $\sigma_i \in \Sigma \setminus J$ and $m+1 \leq i \leq m+s$ respectively. If $X_\Gamma$ is non-compact, we also require that $f$ vanishes at the cusps of $X_\Gamma$. We write $S^J_2(\Gamma)$ for the finite-dimensional complex vector space of weight $2$ type $J$ cuspidal modular forms for $\Gamma$. For a detailed 
discussion, we refer the reader to \cite[Section 2]{hida94}, with two warnings. Firstly, our notion of ``modular form'' corresponds to what in \cite{hida94} is called ``automorphic form''.  Secondly, we prefer to work with 
$ j(\gamma_i, z_i)^{-1}$ at the complex places. This will make the differential forms that we will have to use for the generalized Eichler-Shimura isomorphism that follows below slightly different than those used by Hida.

This definition generalizes the well-known Hilbert and Bianchi modular forms in the weight $2$ case. To see this, let $B$ be the matrix algebra. We get the former by choosing $F$ to be totally real and $J=\Sigma_\R$. We obtain the latter by simply taking $F$ to be imaginary quadratic.

Given an ideal $\fn$ with a totally positive generator $\pi$, let 
$\gamma_\pi$ be as before. We have a finite decomposition $\Gamma \gamma_\pi \Gamma = \sqcup_j \Gamma \gamma_j$ and 
we define the associated Hecke operator $T_\fp$ on $S^J_2(\Gamma)$ via the rule
$$T_\fp(f)(\underline{z}) =  \sum_j f( \gamma_j \cdot  \underline{z} ).$$

We define the involutions at infinity $T_{\sigma_i}$ on modular forms following the above paragraph. The operator $T_{\sigma_i}$ reverses the analytic behaviour of modular forms at the $i$-th (real) component and thus takes $S_2^J(\Gamma)$ to $S_2^{J'}(\Gamma)$ where $J'=J \setminus \{\sigma_i \}$ if $\sigma_i \in J$, and $J'=J \cup \{\sigma_i \}$ otherwise. To illustrate, momentarily take $B$ to be the matrix algebra over $F$. Let us take  $f \in S_2^J(\Gamma)$ and fix some $T_{\sigma_i}$ given by some $\left ( \begin{smallmatrix} \epsilon & 0 \\ 0 & 1 \end{smallmatrix} \right )$ as sketched above. Then we have 
\[
(T_{\sigma_i}f)(\underline{z})= 
f(\epsilon z_1, \hdots,  \epsilon \bar{z_i}, \hdots, 
\epsilon z_{m}, z'_{m+1}, \hdots, z'_{m+s})
\]
where for components $i=m+1,\hdots, m+s$, if $z=(x,y)$ then 
$z'$ means $(\epsilon x ,y)$.
 
\subsection{The Jacquet-Langlands correspondence }
Let us now present the celebrated \emph{Jacquet-Langlands correspondence} in a special case that will be used in the sequel.
Consider two quaternion algebras $B,B'$ over the same number field $F$. For the applications we have in mind in this paper, we assume that $B$ is split (the $2\times 2$-matrix algebra) 
and that $B'$ is non-split (a division ring) with discriminant $\fd$. Let us take some ideal $\fn$ coprime to $\fd$ and take two arithmetic groups 
$\Gamma \subset \GG_B(F) = \GL_2(F)$ and $\Gamma \subset \GG_{B'}(F)$ such that they arise, as above, from the reduced norm $1$ elements of the Eichler orders 
of levels $\fd \fn$ and $\fn$ in $B$ and $B'$ respectively. Given any subset $J$ of the set of real places at which $B'$ is split,  
the Jacquet-Langlands correspondence says that  the spaces $S^J_2(\Gamma)^{\fd{\text -new}}$ and $S^J_2(\Gamma')$ are isomorphic as modules under the 
action of Hecke operators $T_\mathfrak{l}$ and $U_\mathfrak{l}$ ($\mathfrak{l}$ has to be coprime to $\fd$) on both sides. Here, the new subspace is defined in a fashion similar to how we defined it in the cohomological setting above.

\subsection{Generalized Eichler-Shimura Isomorphism}
Given a cuspidal modular form, we can associate cohomology classes to it by forming various types of cuspidal harmonic differential forms. 
Let us fix some cohomology degree $m+s \leq q \leq m+2s$ and 
let $\Sigma_\R$ be as above and $\Sigma_\C$ denote the complex places of $F$. 
Given $J \subset \Sigma_\R$ and $J' \subset \Sigma_\C$ such that $|J'|=m+2s-q$, to every $f \in S^J_2(\Gamma)$, we can associate 
a differential form $\delta_{J,J'}(f)$ on $X$ of degree $q$. The form $\delta_{J,J'}(f)$ is constructed by using the differential $1$-form $dz$ at every real place in $J$ and $d\bar{z}$ at all the 
other real places in $\Sigma_\R$. At every complex place in $J'$, we employ the $\C^3$-valued differential $1$-form 
$dz = ( -dx/y, dy/y, d\bar{x}/y)$, and at all the other complex places we use the differential $2$-form $dz \wedge dz$. More precisely, we have $\delta_{J,J'}(f) = f \bullet dz_{J,J'}$ where
$$dz_{J,J'} = 
 \bigwedge_{\sigma \in J \cup J'} d z_\sigma 
\ \wedge \bigwedge_{\sigma \in \Sigma_\R \setminus J} d \bar{z}_\sigma 
\ \wedge \bigwedge_{\sigma \in \Sigma_\C \setminus J'} dz_\sigma \wedge dz_\sigma  $$

It turns out that $\delta_{J,J'}(f)$ is a harmonic differential form on $X_\Gamma$. The \emph{generalized Eichler-Shimura isomorphism}  (see \cite[Corollary 2.2]{hida94}) says that the map 
\begin{align}\label{eq: Eichler-Shimura}
\bigoplus_{J \subset \Sigma_\R} \bigoplus_{\substack{J' \subset \Sigma_\C \\ |J'|=m+2s-q}} S_2^J(\Gamma) \longrightarrow H^q_{cusp}(\Gamma, \C)
\end{align}
given by 
\begin{align}\label{eq: mod forms to dif} \bigoplus_{J,J'} f_{J,J'} \mapsto  \sum_{J,J'}  \delta_{J,J'}(f_{J,J'}) \end{align}
is an isomorphism of Hecke modules. The details can be found in see \cite[Section 2.4]{hida94}.

\section{Archimedean Darmon points}
\label{section: archimedean darmon points}

We return to the notations and assumptions of \S \ref{construction-main}. In particular, $F$ is a number field of signature $(r,s)$ with $n$ real places splitting in the quadratic extension $K/F$, and $E/F$ an elliptic curve of squarefree conductor $\fN$. In the whole section we assume that there exists at least one real place of $F$ which is ramified in $K$. Recall our labeling of the archimedean places of $F$:
\begin{itemize}
\item $\{\sigma_1,\dots,\sigma_n\}$ are real and split in $K$ (this set is empty if $n=0$);
\item $\{\sigma_{n+1},\dots,\sigma_r\}$ are real and ramify in $K$ (in this \S~this set is non-empty by assumption);
\item $\{\sigma_{r+1},\dots,\sigma_{r+s}\}$ are complex (and the set is empty if $s=0$).
\end{itemize}
We choose as our distinguished place $v\in S(E,K)$ an archimedean place, say $v=\sigma_{n+1}$. Therefore $F_v\cong\R$, $K_v\cong\C$, and $\cH_v=\cH$ (the complex upper half plane). In addition, $B^\times$ acts on the symmetric space
\begin{eqnarray*}
  \cHtwo\times \cHtwo^n\times\cHthree^s
\end{eqnarray*}
by means of $\iota_{\sigma_{n+1}}\times \iota_{1}\times\dots\times\iota_n\times\iota_{r+1}\times\dots\times\iota_{r+s}$ (cf. \eqref{eq: splitting at v} \eqref{eq: splitting at real} \eqref{eq: splitting at complex}). Thanks to our running assumption that $\operatorname{sign}(E/K)=-1$ the set 
\begin{eqnarray}
\label{eq: set of ramification archimedean}
S(E,K)\setminus \{\sigma_{n+1}\} =\{\fq \mid \fN\colon \text{ $\fq$ is inert in $K$}\}\cup \{\sigma_{n+2},\dots, \sigma_{r}\}
\end{eqnarray}
has even cardinality, and $B$ is the quaternion algebra over $F$ that ramifies exactly at \eqref{eq: set of ramification archimedean}. Recall the ideal $\fm$ defined by $\fN = \fd \fm$, where $\fd$ denotes the discriminant of $B$.

We also fix $R_0(\fm)\subset B$ an Eichler order of level $\fm$ and we let $\Gamma = \Gamma_0(\fm)\subset B^\times/F^\times$ be the image of its group of norm $1$ elements under the projection $B\ra B^\times/F^\times$. Let $\cO\subset K$ be an order of conductor relatively prime to $\fN$ and let $\psi\colon \cO\hookrightarrow R_0(\fm)$ be an optimal embedding.

In the remaining of the section, building on the circle of ideas introduced by \cite[Chapter 8]{darmon-book}, \cite{Ga-art}, and \cite{Gr} in the case $s=0$, we give a general construction of Darmon points on $E(\C)$ that also takes into account the case $s>0$. As outlined in the introduction, we break it into three steps: construct a homology class $\Delta_\psi\in H_{n+s}(\Gamma,\Div^0\cHtwo)$ associated to an optimal embedding $\psi$; construct a cohomology class $\Phi_E\in H^{n+s}(\Gamma,\Omega^1_{\cHtwo})$ associated to the character $\lambda_E$; and construct an ``integration pairing'' to define $P_\psi $ in terms of the pairing $\int_{\Delta_\psi}\Phi_E$.

\subsection{The homology class}\label{subsec: homology archimedean}

For any optimal embedding $\psi\in\mathcal{E}(\cO,R_0(\fm))$ the field $K$ acts on $\cHtwo$ via $\iota_{\sigma_{n+1}}\circ\psi$ . Since $\sigma_{n+1}$ ramifies in $K$ there exists a unique $\tau_\psi\in \cHtwo$ fixed by $K$. By the Dirichlet unit theorem the abelian group $\cO_1^\times$, consisting of units of $\cO$ having norm $1$ over $F$, has rank given by
\begin{eqnarray*}
  \operatorname{rk}_\Z(\cO_1^\times) = \operatorname{rk}_\Z (\cO_K^\times)-  \operatorname{rk}_\Z(\cO_F^\times) = 2n+r-n+2s-(r+s)= n+s.
\end{eqnarray*}
In particular, $\cO_1^\times/\text{torsion}\simeq \Z^{n+s}$ so that $
  H_{n+s}(\cO_1^\times/\text{torsion},\Z)\simeq \Z$ and we fix a generator $\Delta$ of this homology group. Let $\Gamma_{\tau_\psi}$ denote the stabilizer of $\tau_\psi$ under the action of $\Gamma$, which  is isomorphic to $  \psi(\cO_1^\times)/\text{torsion}$. Following Greenberg \cite{Gr}, we let $\tilde \Delta_\psi\in H_{n+s}(\Gamma_{\tau_\psi},\Z)$ be an element that maps to $ \Delta$ under the identification $$H_{n+s}(\Gamma_{\tau_\psi},\Z) \simeq H_{n+s}(\cO_1^\times/\text{torsion},\Z).$$

The inclusion $\Gamma_\psi\subset \Gamma$, which is compatible with the degree map $\Div\cHtwo\ra \Z$, induces a natural homomorphism
 \begin{eqnarray}\label{eq: def of j}
   j\colon H_{n+s}(\Gamma_{\tau_\psi},\Z) \longrightarrow H_{n+s}(\Gamma,\Div\cHtwo).
 \end{eqnarray}
In terms of the \emph{bar resolution} (see~\cite[Chapter I.5]{brown1982cohomology}), the element $j (\tilde \Delta_\psi)$ can also be described as follows: let $u_1,\ldots, u_{n+s}$ be a basis of $\cO^\times_1/\text{torsion}$, and set $\gamma_i=\psi(u_i)$. Then:
\[
j (\tilde \Delta_\psi) = \left [ \sum_{\sigma\in\mathfrak{S}_{n+s}} \sgn(\sigma) (\gamma_1|\cdots | \gamma_{n+s}) \tns \tau_\psi \right ].
\]
For example, when $n=s=0$ this is just the class defined by $\tau_\psi$ (c.f. Section~\ref{subsection: HP}), and when $n+s=1$ this is the class defined by $\gamma_1\tns \tau_\psi$, which will be used in the explicit computations of Section~\ref{section: effective methods}.

The exact sequence  
\begin{eqnarray}\label{eq: sex in arch}
  0\lra \Div^0\cHtwo \lra \Div\cHtwo\stackrel{\deg}{\lra} \Z\lra 0
\end{eqnarray}
gives rise to the long exact sequence of homology groups
\begin{eqnarray}
\label{eq:homology exact sequence}
  \cdots \ra H_{n+s+1}(\Gamma,\Z) \stackrel{\delta}{\ra} H_{n+s}(\Gamma,\Div^0\cHtwo)\ra H_{n+s}(\Gamma,\Div\cHtwo)\stackrel{\deg}{\ra} H_{n+s}(\Gamma,\Z)\ra \cdots, 
\end{eqnarray}
where $\delta$ denotes the connection homomorphism. We remark that, since the maps of \eqref{eq: sex in arch} are $B^\times$-equivariant, all the maps of \eqref{eq:homology exact sequence} are Hecke equivariant. Fix a prime $\mathfrak{l}$ relatively prime to $\fN$ and define the operator $\pi=T_\mathfrak{l}-|\mathfrak{l}|-1$, where $|\mathfrak{l}|$ stands for the norm of $\mathfrak{l}$.
\begin{proposition}
The homology class $\deg\circ\pi\circ j(\tilde\Delta_\psi)\in H_{n+s}(\Gamma,\Z)$ is torsion.
\end{proposition}
\begin{proof}
First of all $T_\mathfrak{l}$ does not act on the cuspidal component $H_{n+s}^{\text{cusp}}(\Gamma,\C)$ as $|\mathfrak{l}|+1$ (this follows from the known bounds for eigenvalues of Hecke operators on cusp forms as in, e.g., \cite{luo-rudnick-sarnak}). But $T_\mathfrak{l}$ does act as  $|\mathfrak{l}|+1$ on the universal and infinity subspaces, and this implies that  $\pi(H_{n+s}(\Gamma,\C))\subset H_{n+s}^{\text{cusp}}(\Gamma,\C)$. But $H^{i}_{\text{cusp}}(\Gamma,\C)=0$ for $n+1+s\leq i \leq n+1+2s$ by \eqref{eq: middle degree}. Thus we see that $H_{n+s}^{\text{cusp}}(\Gamma,\C)=0$ because it is the dual to $H^{n+s}_{\text{cusp}}(\Gamma,\C)$. From the commutativity of the diagram \[\xymatrix{
H_{n+s}(\Gamma,\Div \cHtwo) \ar[d]^\pi \ar[r]^{ \text{deg}} &H_{n+s}(\Gamma,\Z)\ar[d]^\pi\\
H_{n+s}(\Gamma,\Div \cHtwo) \ar[r]^{\text{deg}}
&H_{n+s}(\Gamma,\Z)}\]
we obtain that $\deg\circ\pi\circ j(\tilde\Delta_\psi)$ is trivial as an element in $ H_{n+s}(\Gamma,\C)$, from which the result follows.
\end{proof}
Therefore, some multiple of $\pi\circ j(\tilde \Delta_\psi)$ has a preimage in $ H_{n+s}(\Gamma,\Div^0\cHtwo)$ under the map of \eqref{eq:homology exact sequence}. Let $e$ be the smallest such multiple, and define  \[\Delta_\psi\in H_{n+s}(\Gamma,\Div^0\cHtwo)\] to be any preimage of $e\cdot \pi\circ j(\tilde \Delta_\psi)$. 
\begin{remark}
Observe that  $\Delta_\psi$ is in fact only well defined up to elements in $\delta(H_{n+s+1}(\Gamma,\Z))$.
\end{remark}

\begin{remark}\label{rk: geometric construction}  If $K/F$ is not a CM extension, there is a way to pull back (a multiple of) $j(\tilde\Delta_\psi)$ to $H_{n+s}(\Gamma,\Div^0\cHtwo)$, rather than pulling back the ``projection to the cuspidal component'' $\pi\circ j(\tilde\Delta_\psi)$ as we did above. This is more akin to the approach taken by Gartner \cite{Ga-art} and Greenberg \cite{Gr}. Since in the end we will pair the homology class with a cohomology class that lies in the cuspidal component (cf.\S \ref{subsection: IntegrationPairingArchimedean} below), and the integration pairing is Hecke equivariant, the resulting points will be the same. 
\end{remark}

\subsection{The cohomology class}
\label{sec:arch-cohomology}
Recall the character $\lambda_E$ associated to $E$ and to a choice of ``signs at infinity'' in \eqref{eq: lambda_E}. In this section we construct a cohomology class $\Phi_E\in H^{n+s}(\Gamma,\Omega^1_{\cHtwo})$ attached to $\lambda_E$. 

By our running assumption that $E$ is modular and the Jacquet--Langlands correspondence there exists a $\fd$-new modular form $f_E$ for $\Gamma$ such that $T_\mathfrak{l}f_E=a_\mathfrak{l}(E)f_E$ for all primes $\mathfrak{l}\nmid \mathfrak{d}$. Recall that $f_E$ can be regarded as a function
\begin{align*}
  f_E\colon \cHtwo\times \cHtwo^n\times \cHthree^s \lra \C^{3^s}
\end{align*}
which is holomorphic in the $n+1$-variables in $\cHtwo$, harmonic in the $s$ variables in $\cHthree$, and satisfies certain invariance properties with respect to the action of $\Gamma$. Such a modular form is an eigenform for the Hecke operators associated to prime ideals, but not for the involutions at the infinite places. To remedy this situation, we introduce $\tilde f_E$:
\begin{align*}
  \tilde f_E = \sum_{i=1}^n \lambda_E(\sigma_i)T_{\sigma_i}(f_E),
\end{align*}
where $T_{\sigma_i}$ denotes the Hecke operator at the infinite place $\sigma_i$. We see that, by construction, $\tilde f_E$ satisfies that $T_w(\tilde f_E)=\lambda_E(w)\tilde f_E$ at all places $w$ away from $\mathfrak d v$. In other words, $\tilde f_E$ lies in the isotypical component corresponding to the character $\lambda_E$.

We denote by  $\omega_E$ the differential form on the real manifold $\cHtwo\times \cHtwo^n\times \cHthree^s$ induced by $\tilde f_E$ under the map \eqref{eq: mod forms to dif}. It is worth noting that, although this differential form is not holomorphic in general (the hyperbolic $3$-space $\cHthree$ does not even have a complex structure) the form $\omega_E$ is actually holomorphic in the first variable.

To simplify the notation, we set $t=n+s$. The differential form $\omega_E$ gives rise to a cohomology class $\Phi_E\in H^t(\Gamma,\Omega^1_{\cHtwo})$ as follows. Fix a base point $(x_1,\ldots,x_t)\in \cHtwo^n\times\cHthree^{s}$. Define the cochain $k_E\in C^t(\Gamma,\Omega^1_\cHtwo)$ by:
\[
\kappa_E(\gamma_1,\ldots,\gamma_{t})=\int_{x_1}^{\gamma_1 x_1}\int_{\gamma_1 x_2}^{\gamma_1\gamma_2 x_2}\cdots\int_{\gamma_1\cdots\gamma_{t-1}x_t}^{\gamma_1\cdots\gamma_{t-1}\gamma_t x_t} \omega_E,
\]
where the integrals appearing in the right hand side are line integrals on the factors $\cHtwo^n\times\cHthree^s$.
\begin{lemma}
The cochain $\kappa_E$ is an $m$-cocycle.
\end{lemma}
\begin{proof}
We do the case $t=1$, since the general case is done similarly. We need to check that:
\[
\kappa_E(\gamma_1)-\kappa_E(\gamma_1\gamma_2)+\gamma_1\kappa_E(\gamma_2) = 0 
\]
for all  $\gamma_1,\gamma_2\in \Gamma$.
For $P$ and $Q$ arbitrary points in $\cHtwo$ we have
\begin{align*}
\int_P^Q\int_{x_1}^{\gamma_1x_1}\omega-\int_P^Q\int_{x_1}^{\gamma_1\gamma_2x_1}\omega  + \int_{\gamma_1^{-1}P}^{\gamma_1^{-1}Q}\int_{x_1}^{\gamma_2x_1}\omega &= \int_P^Q\int_{x_1}^{\gamma_1x_1}\omega-\int_P^Q\int_{x_1}^{\gamma_1\gamma_2x_1}\omega  + \int_{P}^{Q}\int_{\gamma_1x_1}^{\gamma_1\gamma_2x_1}\omega\\
&=\int_P^Q\int_{x_1}^{\gamma_1\gamma_2x_1}\omega-\int_P^Q\int_{x_1}^{\gamma_1\gamma_2x_1}\omega=0.
\end{align*}
Therefore, the differential form $(d\kappa_E)(\gamma_1,\gamma_2)$ integrates to zero along any path, hence it is zero.
\end{proof}
\begin{lemma}
  The class of $\kappa_E$ in $H^t(\Gamma,\Omega^1_{\cHtwo})$ does not depend on the base point $(x_1,\ldots,x_t)$.
\end{lemma}
\begin{proof}
 We also do the case $t=1$, the general case is done similarly by induction. In this proof we denote by $\kappa_E^{x}$  the cocycle defined using $x$ as the base point. Given $x_1$ and $x_2$,  define a one-form $\alpha$ by:
\[
\alpha = \int_{x_1}^{x_2}\omega_E.
\]
As before, the integral appearing on the right is a line integral on the second factor. Then:
\begin{align*}
\int_P^Q\int_{x_1}^{\gamma x_1}\omega_E - \int_P^Q\int_{x_2}^{\gamma x_2}\omega_E &= \int_P^Q\int_{x_1}^{x_2} \omega_E- \int_P^Q\int_{\gamma x_1}^{\gamma x_2}\omega_E =\int_P^Q\int_{x_1}^{x_2}\omega_E - \int_{\gamma^{-1}P}^{\gamma^{-1}Q}\int_{x_1}^{x_2}\omega_E\\
&=\int_P^Q\alpha - \int_{\gamma^{-1}P}^{\gamma^{-1}Q}\alpha = \int_P^Q(\alpha-\gamma\alpha).
\end{align*}
Therefore we deduce that $\kappa_E^{x_1}-\kappa_E^{x_2}$ is a coboundary.
\end{proof}
Finally, we define $\Phi_E$ as the class of $\kappa_E$ in $H^t(\Gamma,\Omega^1_{\cHtwo})$.

\begin{proposition}
  The class $\Phi_E$ belongs to the $\lambda_E$-isotypical component of $H^t(\Gamma,\Omega^1_{\cHtwo})$.
\end{proposition}
\begin{proof}
  This follows from the explicit description of the action of the Hecke algebra on group cohomology, together with the fact that $\omega_E$ belongs to the $\lambda_E$-isotypical component.
\end{proof}

\begin{remark}
The class $\Phi_E$ constructed above can be related to the cocycle $\kappa'=\kappa'_E$ introduced in~\cite[Section 8.3]{darmon-book}, in the restricted case considered in loc. cit. That is, suppose that $F$ is a totally real field, and  that $K/F$ is an almost totally real extension. Moreover, assume that $\# S(E,K) = 1$. Consider the exact sequence of $\Gamma$-modules
\[
0\to \C \to \cO_{\cHtwo}\tto{d} \Omega^1_{\cHtwo}\to 0,
\]
where $\cO_{\cHtwo}$ denotes the vector space of complex analytic functions on $\cHtwo$, and $d$ is the usual differential map. Here the left action of $\Gamma$ on $\cO_{\cHtwo}$ and $\Omega^1_{\cHtwo}$ is induced from the corresponding action of $\Gamma$ on $\cHtwo$ given by the embedding $\iota_v$. In particular, the action of $\Gamma$ on $\Omega^1_{\cHtwo}$ is given by:
\[
\gamma\cdot f(t)dt = f(\gamma^{-1} t)d\gamma^{-1} t.
\]

 Taking $\Gamma$-invariants yields a long exact sequence:
\begin{equation}
\label{eq:les-cohomology}
\cdots\to H^{n}(\Gamma,\cO_{\cHtwo})\to H^{n}(\Gamma,\Omega^1_{\cHtwo})\tto{\delta'_{n}} H^{n+1}(\Gamma,\C)\to H^{n+1}(\Gamma,\cO_{\cHtwo})\to\cdots
\end{equation}


It is easy to verify that the class $\Phi_E$ satisfies $\delta_n'(\Phi_E)=[\kappa'_E]$.
%

\end{remark}

\subsection{Integration pairing and construction of the point}
\label{subsection: IntegrationPairingArchimedean}

Path integration of holomorphic  $1$-forms on $\cHtwo$ gives a $\Gamma$-equivariant integration pairing 
\begin{eqnarray*}
  \begin{array}{ccc}
    \Div^0\cHtwo \times \Omega_{\cHtwo}^1 & \lra & \mathbb{C}\\
 (\tau_2-\tau_1,\omega) & \longmapsto & \begin{displaystyle} \int_{\tau_2}^{\tau_1}  \omega. \end{displaystyle}
  \end{array}
\end{eqnarray*}
which, by cap product, induces a pairing
\begin{eqnarray*}
  \begin{array}{ccc}
    H_{n+s}(\Gamma,\Div^0\cHtwo) \times H^{n+s}(\Gamma,\Omega_{\cHtwo}^1) & \lra & \mathbb{C}\\
 (\Delta,\omega) & \longmapsto & \begin{displaystyle} \int_\Delta\omega. \end{displaystyle}
  \end{array}
\end{eqnarray*}
\begin{remark}
  Although in \S \ref{construction-main} we phrased the construction in terms of a multiplicative integration pairing, in the archimedean case it seems more natural to use the usual additive integral, and this is how we proceed in this section. One can translate everything to the multiplicative notation by simply applying the exponential map to the line integrals, and working with Tate's uniformization instead of the Weierstrass map.
\end{remark}

Recall the connecting homomorphism  $ \delta\colon H_{n+s+1}(\Gamma,\Z) \ra H_{n+s}(\Gamma,\Div^0\cHtwo)$ arising in~\eqref{eq:homology exact sequence} and define
\begin{eqnarray}
  L = \left\{\int_{\delta(c)} \Phi_E\colon c \in H_{n+s+1}(\Gamma,\Z)\right\}\subset \C.
\end{eqnarray}
Now let $\omega_\cE\in H^0(E,\Omega^1_E)$ be a global regular differential $1$-form on $E$ which extends to a smooth differential on the N\'eron model $\cE/\cO_F$ of $E$. For each $i\in\{1,\dots, r+s\}$ denote by $E_i/\C$ the base change of $E$ by $\sigma_i$, and let $\Lambda_i$ denote the corresponding period lattice obtained by integrating $\sigma_i(\omega_\cE)$ against $H^1(E_i(\C),\Z)$. For $\sigma_i$ real denote by $\Omega_i^+$ (resp. $\Omega_i^-$) a generator of $\Lambda_i\cap \R$ (resp. of $\Lambda_i\cap i \R$). For $\sigma_i$ complex simply denote by $\Omega_i$, $\Omega_i'$ some generators of $\Lambda_i$.  We denote by $\Lambda_v$ the period lattice attached to the distinguished real place $v=\sigma_{n+1}$.  We define
\begin{eqnarray}
  \Omega^{\lambda_E} = \prod_{i=1}^n \Omega_i^{\lambda_E(\sigma_i)}\cdot \prod_{i=r+1}^{r+s}\operatorname{Im}(\Omega_i
\Omega_i')\end{eqnarray}
and the lattice $
  \Lambda_E = \Omega^{\lambda_E}\Lambda_v$. We remark that $\Lambda_E$ is a natural generalization to mixed signature base fields of the lattice considered in \cite[\S 8.3]{darmon-book}. In the same spirit, the following is a generalization of a conjecture of Oda (cf. \cite[Conjecture 2.1]{darmon-logan}) and of Yoshida (cf. \cite[\S 3]{Ga-art}) to the case where $F$ is of mixed signature.
\begin{conjecture}\label{conj: arch}
There exists an isogeny  $\beta\colon \C/L \ra E_v(\C)$.
\end{conjecture}
We remark that the conjecture stated above just asserts the existence of a nonzero complex number $\alpha$ such that $\alpha\Lambda_E\subseteq L$. When $B=M_2(F)$ is the split algebra and $f_E$ is normalized, $\alpha$ can be taken to be algebraic and, in fact can be made explicit up to small rational factors (see the formula~\eqref{eq:qexpansion}). For example, when $F=\Q$ the number $\alpha$ is related to the \emph{Manin constant}.

Granting the conjecture we define
\[
P_\psi = \beta\left(\int_{\Delta_\psi}\Phi_E \right).
\]
We identify $H_\cO^+$ with a subfield of $\C$ by means of a fixed extension of $v$ to $H_{\cO}^+$. The following conjecture is a natural generalization of~\cite[Conjecture 5.1]{Ga-art} and~\cite[Conjecture 8.17]{darmon-book}.
\begin{conjecture}
  The isogeny $\beta$ can be chosen in such a way that $P_\psi$ belongs to $E_v(H_\cO^+)$ and $P_{\alpha\cdot\psi}=\operatorname{rec}(\alpha)(P_\psi)$ for any $\alpha\in\operatorname{Pic}^+(\cO)$.  
\end{conjecture}

\section{Non-archimedean Darmon points}
\label{section: non-archimedean Darmon points}

We return to the notations and assumptions of \S \ref{construction-main}, but now we suppose that there exists at least one prime $\fp \mid \fN$ that is inert in $K$, and we take $v=\fp$ as the distinguished place. Therefore, in this section $F_v=F_\fp$ is a $p$-adic field, $K_v=K_\fp$ is the quadratic unramified extension of $F_\fp$, and $\cH_\fp=K_\fp\setminus F_\fp$. 

Recall the factorization of the conductor of $E$, which now takes the form
\begin{eqnarray}
  \fN = \fp\fd\fm,
\end{eqnarray}
where $\fd$ is the product of primes different from $\fp$ that are inert in $K$ and  $\fm$ is the product of primes that are split in $K$. The quaternion algebra $B$ ramifies at the places
\begin{eqnarray}\label{eq:ramif non-arc} 
\{\fq \mid \fd\}\cup \{\sigma_{n+1},\dots, \sigma_{r}\},
\end{eqnarray}
and it acts  on 
\begin{eqnarray*}
  \cH_\fp\times \cHtwo^{n}\times \cHthree^{s}
\end{eqnarray*}
by means of $\iota_{\fp}\times \iota_{1}\times\dots\times\iota_n\times\iota_{r+1}\times\dots\times\iota_{r+s}$ (cf. \eqref{eq: splitting at v} \eqref{eq: splitting at real} \eqref{eq: splitting at complex}). We fix also $R_0(\fp\fm)\subset R_0(\fm)$ Eichler orders of level $\fp\fm$ and $\fm$, and we set $\Gamma_0(\fp\fm)$ (resp. $\Gamma_0(\fm))$ to be the image of $R_0(\fp\fm)^\times_1$  (resp. of $R_0(\fm)^\times_1$) in $B^\times/F^\times$. Finally, we let $R=R_0(\fm)\otimes_{\cO_F}\cO_{F,\{\fp\}}$ and let $\Gamma$ be the image of $R_1^\times$ in $B^\times/F^\times$.

\subsection{The homology class}

Let $\cO\subset K$ be an order of conductor relatively prime to $\fN$ and let $\psi\in \cE(\cO,R_0(\fm))$ be an optimal embedding. Note that such  embeddings exist because all the primes dividing $\fm$ split in $K$. In this section we construct an homology class $\Delta_\psi\in H_{n+s}(\Gamma,\Div^0\cH_\fp)$. We remark that, unlike in the archimedean setting of \S\ref{subsec: homology archimedean}, now $\Gamma$ is an $S$-arithmetic group which acts on the $\fp$-adic upper half plane $\cH_\fp$. In spite of this difference, the definition of $\Delta_\psi$ is formally very similar to that of \S\ref{subsec: homology archimedean}.

First of all $K$ acts on $\cH_p$ via $\iota_\fp\circ \psi$, and there exists a unique $\tau_\psi\in\cH_\fp$ such that
\[
\iota_\fp\circ\psi(x)\left( \begin{array}{c} \tau_\psi\\ 1\end{array} \right) = x\cdot \left( \begin{array}{c} \tau_\psi\\ 1\end{array} \right) \ \text{ for all $x\in K.$}
\]
Let $\cO_1^\times=\{x\in\cO\colon \Nm_{K/F}x=1\}$ and let $\Gamma_0(\fm)_{\tau_\psi}$be the stabilizer of $\tau_\psi$ in $\Gamma_0(\fm)$. The embedding $\psi$ induces an identification $\Gamma_0(\fm)_{\tau_\psi}\simeq \cO_1^\times/\text{torsion}$. By Dirichlet's unit theorem  $\cO_1^\times$ has rank $n+s$, and we fix a generator $\Delta$ of $H_{n+s}(\cO_1^\times/\text{torsion},\Z)\simeq \Z$. Then we let $\tilde\Delta_\psi\in H_{n+s}(\Gamma_0(\fm)_{\tau_\psi},\Z)$ be an element mapping to $\Delta$ under the identification $H_{n+s}(\Gamma_0(\fm)_{\tau_\psi},\Z)\simeq H_{n+s}(\cO^\times_1/\text{torsion},\Z)$.

As in \eqref{eq: def of j} the inclusion $\Gamma_0(\fm)_{\tau_\psi}\subset \Gamma_0(\fm)$ induces a natural map
\begin{eqnarray*}
  j\colon H_{n+s}(\Gamma_0(\fm)_{\tau_\psi},\Z)\lra H_{n+s}(\Gamma_0(\fm),\Div\cH_\fp).
\end{eqnarray*}
Set $c_\psi=j(\tilde\Delta_\psi)$. In this case, we also need to consider the restriction map induced by the inclusion $\Gamma_0(\fm)\subset \Gamma$:
\begin{eqnarray*}
  \res\colon H_{n+s}(\Gamma_0(\fm),\Div\cH_\fp)\lra H_{n+s}(\Gamma,\Div\cH_\fp).
\end{eqnarray*}
One can get an expression for $\res (c_\psi)\in H_{n+s}(\Gamma,\Div\cH_\fp)$ in terms of the bar resolution which is very similar to that found in~\S\ref{subsec: homology archimedean}. 

The exact sequence of $B$-modules
\begin{align*}
  0\lra \Div^0\cH_\fp \lra \Div\cH_\fp \stackrel{\deg}{\lra}\Z\lra 0
\end{align*}
gives rise to the long exact sequence in $\Gamma$-homology
\begin{align}\label{eq: lex non-arch}
    \cdots \lra H_{n+s+1}(\Gamma,\Z) \stackrel{\delta}{\lra} H_{n+s}(\Gamma,\Div^0\cH_\fp)\lra H_{n+s}(\Gamma,\Div\cH_\fp)\stackrel{\deg}{\lra} H_{n+s}(\Gamma,\Z)\lra \cdots.
\end{align}
We fix a prime $\mathfrak{l}\nmid \fN$ and denote by $\pi$ the element of the Hecke algebra $T_\mathfrak{l}-|\mathfrak{l}|-1$, which we will regard as an operator acting on different homology groups. 
\begin{proposition}
The homology class $\deg \circ \pi\circ \res (c_\psi)\in H_{n+s}(\Gamma,\Z)$ is torsion.
\end{proposition}
\begin{proof}
  By the commutativity of the diagram 
\[\xymatrix{
H_{n+s}(\Gamma,\Div \cHtwo) \ar[d]^\pi \ar[r]^{ \text{deg}} &H_{n+s}(\Gamma,\Z)\ar[d]^\pi\\
H_{n+s}(\Gamma,\Div \cHtwo) \ar[r]^{\text{deg}}
&H_{n+s}(\Gamma,\Z)}\] 
we see that $\deg\circ \pi \circ \res(c_\psi)=\pi\circ \deg\circ \res (c_\psi)$. The following diagram also commutes:
\[\xymatrix{
H_{n+s}(\Gamma_0(\fm),\Div \cHtwo) \ar[d]^\res \ar[r]^{ \deg} &H_{n+s}(\Gamma_0(\fm),\Z)\ar[d]^\res\\
H_{n+s}(\Gamma,\Div \cHtwo) \ar[r]^{\deg}
&H_{n+s}(\Gamma,\Z)}\]
and we see that $\deg\circ \pi \circ \res(c_\psi) = \pi\circ \res\circ\deg (c)$. By Theorem~\ref{blasius}, the image of the restriction map $\res \colon H_{n+s}(\Gamma_0(\fm),\C)\ra H_{n+s}(\Gamma,\C)$ is isomorphic to $H_{n+s}^{\text{univ}}(\Gamma_0(\fm),\C)$. Since 
Hecke operators commute with restriction maps and $\pi \left ( H_{n+s}^{\text{univ}}(\Gamma_0(\fm),\C) \right ) = 0$, we have that $\pi\circ \res \circ \deg (c_\psi)$ is $0$, viewed as an element of $H_{n+s}(\Gamma,\C)$ and this proves the proposition.
\end{proof}

Therefore, in view of \eqref{eq: lex non-arch}, there is a positive multiple of $\pi\circ\res (c_\psi)$ lying in the image of $H_{n+s}(\Gamma,\Div^0\cH_\fp)$. Let $e$ be the smallest such multiple, and we define  \[\Delta_\psi\in H_{n+s}(\Gamma,\Div^0\cH_\fp)\] to be the element mapping to $e\cdot \pi\circ\res (c_\psi)$.
\begin{remark}
  The element $\Delta_\psi$ is only well defined up to elements in $\delta(H_{n+s+1}(\Gamma,\Z))$.
\end{remark}
\begin{remark}
  Similarly as in Remark \ref{rk: geometric construction}, if $K/F$ is not CM one can also show directly that a multiple of $\res\circ j(\tilde\Delta_\psi)$ has a preimage in $H_{n+s}(\Gamma,\Div^0\cH_\fp)$. This requires introducing a geometric interpretation of the cycles involved, and it is what is done in \cite[\S 7]{Gr} in the case of $F$ being totally real. If $F$ has complex places the same argument goes through, with some modifications due to the factors $\cHthree$ in the symmetric domain.
\end{remark}

\subsection{The cohomology class}
\label{subsection: non_archimedean cohomology class}

\newcommand{\red}{\operatorname{red}}
We proceed to define the cohomology class attached to the character $\lambda_E$ of \eqref{eq: lambda_E}. Recall the Bruhat-Tits $\cT$ of $\GL_2(F_\fp)$. It is a $(|\fp|+1)$-regular tree which is a deformation retract of the $\fp$-adic upper half plane (see, e.g., \cite{aws2007-dasgupta-teitelbaum} for more details). We denote by $\cV$ (resp. $\cE$) the set of vertices (resp. directed edges) of $\cT$. Given $e\in \cE$ we denote by $\bar e$ its opposite edge, and by $s(e)$ (resp. $t(e)$) its source (resp. target) vertex. For any subring $A$ of $K_\fp$ we denote by $\cF_0(\cE,A)$  the $A$-module of functions $f$ from $\cE$ to $A$ satisfying $f(\ol e)=-f(e)$, and by $\cF(\cV,A)$ the $A$-module of functions from $\cV$ to $A$. The module of $A$-valued harmonic cocycles $\HC(A)$ is defined by means of the exact sequence
\begin{align}
  \label{eq:ses-hc}
  0\to \HC(A)\to \cF_0(\cE,A)\tto{h} \cF(\cV,A)\to 0,
\end{align}
 where $h$ is the map
\[
h(f)(v)=\sum_{s(e)=v} f(e).
\]
\begin{proposition}
\label{prop:ses-oneforms}
There is a Hecke-equivariant long exact sequence
\begin{align}\label{eq: lex greenberg}
\cdots\to H^i(\Gamma,\HC(A))\to H^i(\Gamma_0(\fp\fm),A)\tto{\alpha_i} H^i(\Gamma_0(\fm),A)^2\to\cdots
\end{align}
\end{proposition}
The proof of the above proposition can be found in~\cite[\S 8]{Gr} for totally real $F$, but the same argument works in general. We briefly recall the main steps. First of all, taking $\Gamma$-invariants in \eqref{eq:ses-hc} one obtains the long exact sequence
\begin{align*}
\cdots \to H^{i}(\Gamma,\HC(A))\to H^{i}(\Gamma,\cF_0(\cE,A))\to H^{i}(\Gamma,\cF(\cV,A))^2\to \cdots.
\end{align*}
If $H$ is a subgroup of a group $G$ acting on $A$ we set $\Ind_H^G A= \Hom_{\ZZ[H]}(\ZZ[G],A)$. Then we have the following identifications
\begin{align*}
  \cF(\cV,A)\cong \left(\Ind_{\Gamma_0(\fm)}^\Gamma A\right)^2 \ \text{ and } \ 
\cF_0(\cE,A)\cong \Ind_{\Gamma_0(\fp\fm)}^\Gamma A.
\end{align*}
Indeed, strong approximation and the fact that the base field $F$ has class number one implies that $\Gamma$ acts on $\cE$ with two orbits $\cE_0$ and $\cE_1$ (usually called \emph{even} and \emph{odd} edges), such that $e\in \cE_0$ if and only if $\bar e\in \cE_1$. The group $\Gamma$ also acts with two orbits on  $\cV$. This gives bijections of $\Gamma$-sets
\[
\cV\cong (\Gamma_0(\fm)\backslash \Gamma)^2,\quad \cE_0\cong \Gamma_0(\fp\fm)\backslash \Gamma,
\]
which induce the claimed isomorphisms. Now the Shapiro isomorphisms
\begin{align*}
H^i(\Gamma,\Ind_{\Gamma_0(\fm)}^\Gamma)\cong H^i(\Gamma_0(\fm),A)\ 
\text{ and }\ 
H^i(\Gamma,\Ind_{\Gamma_0(\fp\fm)}^\Gamma)\cong H^i(\Gamma_0(\fp\fm),A)
\end{align*}
give the exact sequence of Proposition~\ref{prop:ses-oneforms}. 

From \eqref{eq: lex greenberg} we extract the short exact sequence
\begin{equation}
\label{eq:ses-coh-final}
0\to\coker\alpha_{n+s-1}\to H^{n+s}(\Gamma,\HC(A))\to H^{n+s}(\Gamma_0(\fp\fm),A)^{\fp\text{-new}}\to 0.
\end{equation}
The following result is \cite[Proposition 25]{Gr}.
\begin{proposition}
The short exact sequence~\eqref{eq:ses-coh-final} induces an isomorphism
\[
H^{n+s}(\Gamma,\HC(\Q))^{\lambda_E}\cong \left(H^{n+s}(\Gamma_0(\fp\fm),\Q)^{\fp\text{-new}}\right)^{\lambda_E},
\]
where the superscript $\lambda_E$ denotes the $\lambda_E$-isotypical component.
\end{proposition}
As a consequence of the modularity of $E$, the Jacquet--Langlands correspondence, and multiplicity-one \cite{ramakrishnan}, we have that $\left(H^{n+s}(\Gamma_0(\fp\fm),\Q)^{\fp\text{-new}}\right)^{\lambda_E}$ is one-dimensional. Therefore, as a corollary of the above proposition so is $H^{n+s}(\Gamma,\HC(\Q)^{\fp\text{-new}})^{\lambda_E}$.

Denote by $\Omega^1_{\cH_\fp}(\Z)$ the $\Z$-module of rigid-analytic $1$-forms on $\cH_\fp$ for which all of their residues are in $\Z$. There is a reduction map (see~\cite[Section 5.1]{darmon-book}) $\red\colon\cH_\fp\to \cT$, and for every edge $e$ of $\cT$ we denote by $A_e=\red^{-1}(e)$, which is an oriented annulus in $\cH_\fp$. One has $H^1_{\text{dR}}(A_e)\cong \C_p\frac{dz}{z}$, and therefore there is a well defined residue map $\res_{A_e}$ sending $dz/z$ to $1$. An oriented edge $e\in\cE(\cT)$ determines also a compact open subset $U_e\subseteq\P^1(F_\fp)$, and this assignment is $\GL_2(F_\fp)$-equivariant and can be arranged to send a distinguished edge $e_*$ to $\Z_p\subset \P^1(F_\fp)$. Finally, the collection $\{U_e\colon e\in\cE(\cT)\}$ constitutes a basis of compact opens of $\P^1(F_\fp)$. Since any compact open subset $U\subset \P^1(F_\fp)$ can be written as a finite disjoint union of sets of the form $U_e$, the map  $U_e\mapsto \res_{A_e}(\omega)$ can extends to a map $U\mapsto \res_U(\omega)$, with $U\subset \P^1(F_\fp)$ any compact open subset.

\begin{theorem}[Amice--Velu, Vishik]
  The map sending $\omega$ to the harmonic cocycle $e\mapsto \res_{A_e}(\omega)$ induces an isomorphism
\[
\Omega^1_{\cH_\fp}(\Z)\stackrel{\simeq}{\lra} \HC(\Z).
\]
\end{theorem}
\begin{proof}
  This follows at once from~\cite[Corollary 2.3.4]{aws2007-dasgupta-teitelbaum}, and the paragraphs leading to the result provide a very detailed explanation.
\end{proof}
By the universal coefficients theorem $H^{n+s}(\Gamma,\HC(\Q))=H^{n+s}(\Gamma,\HC(\Z))\tns\Q$, and using the above theorem  we may define
\[
\Phi_E\in \left( H^{n+s}(\Gamma,\Omega^1_{\cH_\fp}(\Z))\right)^{\lambda_E}
\]
as a generator of the torsion-free part.

\subsection{Integration pairing and construction of the point}
\label{subsection: non-archimedean integration}

\label{sec:integration-nonarch}
We want to describe an integration theory for rigid analytic $1$-forms. This is essentially the theory of Coleman integration. We are also interested the multiplicative variant, which as mentioned in Section \ref{subsection: HP} appears also in the explicit \cerednikdrinfeld{} uniformization (cf. \cite[\S 5]{darmon-book}). 

Consider the bilinear pairing
\[
\int\colon \Div^0(\cH_\fp)\times \Omega^1_{\cH_\fp}(\Z) \to K_\fp,\quad (D,\omega)\mapsto \int_D \omega
\]
given as follows: for
$
D=\sum_{i=1}^k y_i-x_i$ with $ x_i,y_i\in\cH_\fp$ 
and  $\omega\in \Omega^1_{\cH_\fp}(\Z)$ then
\[
\int_D\omega=\sum_{i=1}^k \int_{x_i}^{y_i} \omega,
\]
where the integrals in the right hand side are Coleman integrals of the rigid form $\omega$ (see for example \cite{ColemanDilogs}). The change of variables formula implies that the integration pairing is $\GL_2(F_\fp)$-equivariant:
\[
\int_{\gamma D}\omega =\int_{D}\gamma\omega,\quad \gamma\in\GL_2(F_\fp).
\]
Coleman integration in $\cH_\fp$ can be described in more elementary terms as follows (see~\cite{teitelbaum} for more details).
\begin{theorem}[Teitelbaum]
  Let $\omega\in \Omega^1_{\cH_p}(\Z)$. Then
\[
\int_x^y \omega = \lim_{\cU}\sum_{U\in \cU}\log_\fp\left(\frac{t_U-y}{t_U-x}\right)\res_{U}(\omega),
\]
where the limit is taken over increasingly fine covers $\cU$ of $\P^1(F_\fp)$ by disjoint compact opens, and $t_U\in U$ is any sample point.
\end{theorem}
There is a multiplicative refinement which replaces the Riemann sums with ``Riemann products''. In this case, if $x$ and $y$ belong to $\cH_\fp$ it is given by:
\[
\Xint\times_x^y\omega = \lim_{\cU}\prod_{U\in\cU}\left(\frac{t_U-y}{t_U-x}\right)^{\res_{U}(\omega)}\in K_\fp^\times.
\]
Observe that the residues of $\omega$ appear as exponents; here is where the assumption that this residues are integers is essential.

Cap product induces a bilinear pairing
\[
\xymatrix@C15pt@R1pt{
H_i(\Gamma,\Div^0(\cH_\fp))\times H^i(\Gamma ,\Omega^1_{\cH_\fp}(\Z))\ar[r]&K_\fp^\times
}
\]
which we denote by
\[
(\Delta,\omega)\mapsto\displaystyle\Xint\times_\Delta\omega.
\]
It is equivariant for the Hecke action, namely
\begin{align}\label{eq: integration pairing is equivariant}
\Xint\times_{T_\mathfrak{l}\Delta} \omega=\Xint\times_{\Delta}T_\mathfrak{l}\omega.
\end{align}
Recall the connection homomorphism $\delta$ of  \eqref{eq: lex non-arch}, and the cohomology class $\Phi_E$ defined in \S \ref{subsection: non_archimedean cohomology class}. Then we let
\[L=\left\{\displaystyle\Xint\times_{\delta(c)}\Phi_E \colon c\in
H_{n+s+1}(\Gamma,\Z) \right\}\subset K_\fp^\times.\]

Exactly the same argument as in~\cite[Section 11]{Gr} shows that $L$ is a lattice of $K_\fp^\times$. The following conjecture is a natural generalization of theorems of Darmon \cite{Da1}, Dasgupta--Greenberg \cite{greenberg-dasgupta}, and \cite{LRV} when $F=\Q$, and of a conjecture of Trifkovi\'c \cite{trifkovic} when $F$ is quadratic imaginary.
\begin{conjecture}
  The lattice $L$ is commensurable to the Tate lattice $\langle q_E\rangle$ of $E/K_\fp$.
\end{conjecture}
Granting this conjecture one can find an isogeny  $\beta: K_\fp^\times/L\ra E(K_\fp)$, and we define 
\begin{align*}
  P_\psi = \beta \left(\Xint\times_{\Delta_\psi}\Phi_E \right).
\end{align*}
 Since $\fp$ splits completely in $H_\cO^+$ we can, and do, regard $H_\cO^+$ as a subfield of $K_\fp$ by choosing a prime in $H_\cO^+$ above $\fp$. The following conjecture is the natural generalization of~\cite[Conjecture 3]{Gr}, \cite[Conjecture 7]{darmon-integration} and~\cite[Conjecture 6]{trifkovic}.
\begin{conjecture}
The local point $P_\psi$ is a global point. More
precisely, it is rational 
over $H^+_\cO$ and for any $\alpha\in\operatorname{Pic}^+(\cO)$ we have that $P_{\alpha\cdot\psi}=\operatorname{rec}(\alpha)(P_\psi)$.
\end{conjecture}

\section{Effective methods and numerical evidence}
\label{section: effective methods} 

In the case of curves defined over totally real fields $F$, so far the most compelling evidence for the conjectural rationality of Darmon points comes from explicit numerical verifications (cf. \cite{darmon-green}, \cite{darmon-pollack}, \cite{guitart-masdeu-h}, \cite{shpquat}, \cite{darmon-logan}, and \cite{guitart-masdeu}). In this section we describe how the new constructions of Darmon points introduced in the present article, in which the curves are defined over fields $F$ of mixed signature, can be performed in certain cases. In addition, we present some explicit computations that allowed us to numerically test the rationality of some of these new instances of points. Such calculations give strong evidence for the validity our conjectures.

In order to compute the points explicitly it is crucial to dispose of algorithms for working with the groups $H_{n+s}(\Gamma,\Div^0\cH_v)$ and $H^{n+s}(\Gamma,\Omega_{v})$. 
We restricted the computations to the case $n+s\leq 1$ since this simplifies the (co)homological computations. It would be interesting to generalize these algorithms to higher (co)homological degrees, in the spirit of~\cite{computations-modforms}.

There are several combinations of extensions $K/F$ and curves $E$ satisfying $n+s\leq 1$ and leading to new instances of Darmon points for which we have been able to perform explicit computations. Before presenting the algorithms in detail in the next subsections,  we record here the cases in which we have computed new instances of Darmon points. For the sake of completeness, and with the aim of giving a unified vision of the computational picture so far, we include also the previously computed cases of Darmon and Heegner points. In this listing we also distinct between archimedean and non-archimedean points, and between the underlying quaternion algebra $B$ being split or division, for the computational techniques are usually rather different depending on these parameters.

In what follows we continue with the notation of the previous sections: the curves are defined over a number field $F$ with signature $(r,s)$, and $n \leq r$ denotes the number of real places that split in $K$.
\begin{enumerate}
\item $n=0$, $s=0$. In this case $K/F$ is a quadratic CM extension and hence one is computing Heegner points.
  \begin{itemize}
  \item \emph{Archimedean computations}. There are a number of calculations of Heegner points exploiting the archimedean uniformization of classical modular curves (this corresponds to $B\simeq \M_2(\Q)$). See for instance \cite{elkies-hp} for an efficient algorithm. 

As for Heegner points on Shimura curves attached to division algebras Elkies  \cite{elkies-CM} performed some computations in certain particular situations. More recently, Voight--Willis \cite{VW} using Taylor expansions  and Nelson \cite{nelson} using the Shimizu lift have provided more general algorithms and have performed some computations.

   \item \emph{ Non-archimedean computations}. Greenberg provided an algorithm and computations \cite{greenberg-hp}, based on the explicit \cerednikdrinfeld{} uniformization of Shimura curves introduced by Bertolini--Darmon \cite{bertolini-darmon-p-adic-periods} and on the overconvergent modular symbols techniques of Pollack--Stevens \cite{PS}.
  \end{itemize}
\item $n=1$, $s=0$. In this case $F$ is totally real and $K$ has exactly two real places. 
  \begin{itemize}
  \item \emph{Archimedean computations}. Explicit computations where carried out in \cite{darmon-logan} and \cite{guitart-masdeu} for real quadratic $F$ and $B\simeq \M_2(F)$. So far there are no computations in case where $B$ is division or $[F\colon \Q]>2$.
   \item \emph{Non-archimedean computations}. When $F=\Q$ and $K$ is real quadratic the computations where performed in \cite{darmon-green} and \cite{darmon-pollack} (see also \cite{guitart-masdeu-h}) for $B\simeq \M_2(\Q)$, and in \cite{shpquat} for  $B$ a division algebra.
  \end{itemize}
\item $n=0$, $s=1$, $r=0$. In this case $F$ is a quadratic imaginary field and $K$ is any quadratic extension. 
  \begin{itemize}
\item \emph{Archimedean computations}. In this setting there is no archimedean construction of Darmon points, because all archimedean places of $F$ split in $K$.
  \item \emph{Non-archimedean computations}. When $B\simeq\M_2(F)$ the construction of the points and the explicit calculations are due to Trifkovi{\'c} \cite{trifkovic}. 

The case where $B$ is division is one of the new features of the present note (they are a particular case of the construction of Section \ref{section: non-archimedean Darmon points}). We present an explicit computation in Example \ref{ex: non-arch 2} which gives numerical evidence in support of the rationality conjecture in this case.

  \end{itemize}
\item $n=0$, $s=1$, $r>0$. In this case $F$ is ATR, and $K$ is a totally complex extension. Since $F$ is of mixed signature, the Darmon points in this setting are a novelty of the present article. We will devote the rest of the section to present the explicit algorithms and computations in the following cases:
  \begin{itemize}
  \item \emph{Archimedean computations}. In \S \ref{subsection: archimedean computations} we present the methods for computing when $F$ is a cubic field of signature $(1,1)$ and $B\simeq\M_2(F)$. An explicit calculation is given in Example \ref{example: explicit computations and examples-arch}.
   \item \emph{Non-archimedean computations}. In \S \ref{subsection: non-archimedean computations} we provide the algorithms for $F$ a field of signature $(r,1)$ and $B$ a division algebra. In Example \ref{ex: non-arch 1} we present an explicit calculation with $r=1$. 
  \end{itemize}
\end{enumerate}

\subsection{Archimedean computations}
\label{subsection: archimedean computations}

Suppose that $F$ is a cubic field of signature $(1,1)$. Let $E/F$ be an elliptic curve of conductor $\fN$, and let $K$ be a totally complex quadratic extension (hence $n=0$) such that all primes dividing $\fN$ split in $K$. We take $R_0(\fN)$ to be the Eichler order in $\M_2(F)$ formed by the upper triangular matrices modulo $\fN$, and we set $\Gamma=\Gamma_0(\fN)$. 

Let $\psi$ be an optimal embedding $\psi\colon \cO_F\hookrightarrow R_0(\fN)$. Fix a unit $u\in(\cO_F)^\times_1$ and set $\gamma_\psi=\psi(u)\in\Gamma$. Let also $\tau_\psi\in\cHtwo$ be a fixed point by the action of $\psi(K)$. Then the cycle $\gamma_\psi\otimes \tau_\psi$ defines a homology class $\tilde\Delta_\psi\in H_1(\Gamma,\Div\cHtwo)$. Applying the techniques of~\cite{guitart-masdeu-h} one can rewrite the original $\tilde\Delta_\psi$ into a cycle with values in $\Div^0\cHtwo$, which will give the homology class $\Delta_\psi\in H_1(\Gamma,\Div^0\cHtwo)$ as defined in Section~\ref{rk: geometric construction}. For the convenience of the reader, we briefly recall how this is done.

We need to consider the subgroup $\Gamma_1(\fN)\subset \Gamma$, defined as:
\[
\Gamma_1(\fN) = \{ \smtx abcd \in \SL_2(\cO_F)\colon \smtx abcd \equiv \smtx 1{ * }01 \pmod{\fN}\}.
\]
Since $\Gamma_1(\fN)$ is contained in $\Gamma$ with finite index, we may assume (after replacing $\gamma_\psi$ with a power of it if necessary) that $\gamma_\psi$ belongs to $\Gamma_1(\fN)$. In terms of homology classes, this amounts to considering a multiple of $\gamma_\psi\tns\tau_\psi$, which will yield a multiple of our desired point.

An elementary matrix for $\Gamma_1(\fN)$ is a matrix of the form:
\[
u(x)=\smtx 1x01,\quad x\in\cO_K,
\]
or
\[
\ell(x) = \smtx 10y1,\quad y\in \fN.
\]
The congruence subgroup property implies that  $\Gamma_1(\fN)$ is generated by elementary matrices. The following result gives an effective version of this result.
\begin{theorem}[{\cite[Theorem 2.3]{guitart-masdeu-h}}]
  Let $\gamma\in\Gamma_1(\fN)$. Assuming a Generalized Riemann Hypothesis, there is an explicit algorithm that finds a decomposition of $\gamma$ as a product of $5$ elementary matrices.
\end{theorem}

The next step is to use a decomposition of $\gamma_\psi$ into elementary matrices to rewrite $\gamma_\psi\tns \tau_\psi$ as a sum of $1$-cycles with values in divisors of degree $0$. The following lemma is useful to this purpose:
\begin{lemma}
Let $g\in \Gamma$. Let $u$ be an upper triangular elementary matrix and let $\ell = s^{-1} u_\ell s$ be a lower triangular elementary matrix, where $s = \smtx 0{-1}{1}{0}$. Then for any divisor $D\in \Div \cHtwo$, we have:
  \begin{itemize}
  \item $ug\tns D \equiv u\tns D + g\tns u^{-1}D$, and
  \item $\ell g\tns D \equiv s\tns(u_\ell^{-1}sD-sD) + g\tns \ell D + u_\ell\tns sD$, where $u_\ell=s^{-1}\ell s$ is upper-triangular.
  \end{itemize}
Here, by the symbol $\equiv$ we mean that the two quantities differ by a $1$-boundary in $B_1(\Gamma,\Div\cHtwo)$.
\end{lemma}

The above lemma, together with a given decomposition of $\gamma_\psi$ into elementary matrices, allows us to write the following equality of $1$-homology classes with values in $\cHtwo\cup\{\infty\}$:
\begin{align}
\label{eq:cycle-decomp}
\gamma_\psi\tns\tau_\psi &= s\tns D_0 + \sum_i u_i\tns (D_i-n_i\infty) + \sum_i n_iu_i\tns \infty
\end{align}
where $D_0$ is a divisor of degree $0$ and $\deg(D_i)=n_i$.
\begin{remark}
 Although the point $\infty$ does not belong to $\cHtwo$, we can use the above formulation. The reason is that the class $\Phi_E\in H^n(\Gamma,\Omega^1_{\cHtwo})$ extends to a class in $H^n(\Gamma,\Omega^1_{\cHtwo\cup\{\infty\}})$, because of the growth condition at infinity. The pairing $\Div^0\cHtwo\times \Omega^1_{\cHtwo}$ obviously extends to a pairing $\Div^0(\cHtwo\cup\{\infty\})\times \Omega^1_{\cHtwo\cup\{\infty\}}$. 
\end{remark}

Since $\sum_i n_i u_i\tns \infty$ is a $1$-cycle, so is the sum of the other two terms of the right-hand side of~\eqref{eq:cycle-decomp}. One easily checks that the only term contributing to the integration pairing is $s\tns D_0$. This is so because one can take as base point to define the cocycle $\Phi_E$ of \S \ref{sec:arch-cohomology} the point $\infty$, which is stabilized by upper-triangular matrices.

On the other hand, to the elliptic curve $E$ there is attached an automorphic form as in \S \ref{prop:ses-oneforms}. Let us denote by $v_0$ the real embedding of $F$, and by $v_1$, $v_2$ the 
complex embeddings. For $\alpha\in F$ we use the notation $\alpha_i = v_i(\alpha)$.

The $L$-series of $E/F$ is of the form 
\[
 L(E/F,s)=\sum_{\fn \subset \cO_F} a_\fn |\fn|^{-s},
\]
and the coefficients $a_{\fn}$ can be computed in practice by counting points on the reductions of $E$ modulo the different primes $\fp$ of $F$.

 By our assumption that $E$ is modular, there is a differential form $\omega_E\in H^2(\Gamma\backslash 
\cHtwo\times \cHthree,\C)$ attached to $E$, which can be explicitly described in terms of its Fourier-Bessel expansion as follows. 
Let $z$ denote the coordinate in $\cHtwo$ and $(x,y)$ with $x\in\C$ and $y>0$ the coordinates in $\cHthree$, and let $\delta$ be a 
positive generator of the different ideal of $F$. Then $\omega_E=\tilde f_E\cdot \beta$, where $\beta$ is the differential in $\cHtwo\times\cHthree$
\[
\beta = \left(\frac{-dx}{y}\wedge d\bar z,\frac{dy}{y}\wedge d\bar z,\frac{d\bar 
x}{y}\wedge d\bar z\right).
\]
The transformation property of Equation~\eqref{eq:cuspidal-modular-form} satisfied by $\tilde f_E$ makes $\omega_E$ invariant under the action of $\Gamma$ and thus defines an element in $H^2(\Gamma\backslash \cHtwo\times\cHthree,\C)$. It has the following expression:
\begin{equation}
\label{eq:qexpansion}
\omega_E(z,x,y) =  \sum_{\substack{\alpha \in \cO_F\\ \alpha_0 >0}} \frac{a_{(\alpha)}}{N_{F/\Q}(\alpha)}\frac{\alpha_0}{\delta_0} \exp\left({-2\pi 
i}\left(\frac{\alpha_0 \bar z}{\delta_0} + \frac{\alpha_1 x}{\delta_1} + \frac{\alpha_2 \bar x}{\delta_2} \right)\right) 
\mathbb{K}\left(\frac{ \alpha_1 y}{\delta_1} \right)\cdot\beta,
\end{equation}
where
\[
\mathbb{K}(t) =\left(-\frac{i}{2}t|t|K_1(4\pi |t|),|t|^2K_0(4\pi|t|),\frac{i}{2}\ol t |t| K_1(4\pi|t|)\right),
\]
and $K_0$ and $K_1$ are the hyperbolic Bessel functions of the second kind:
\[
K_0(t)=\int_0^\infty e^{-t\cosh(h)}dh,\quad K_1(t)=\int_0^\infty e^{-t\cosh(h)}\cosh(h)dh.
\]

This expression can essentially be found in~\cite{hida94}, although the notation is greatly simplified to our very special setting. Let us write the three components of $\omega_E$ as
\[
\omega_E = F_1 \frac{dx}{y}\wedge d\bar z + F_2 \frac{dy}{y}\wedge d\bar z + F_3 \frac{d\bar x}{y}\wedge d\bar z.
\]
We are interested in computing integrals of the form
\[
\int_{z_1}^{z_2}\int_{(x_1,y_1)}^{(x_2,y_2)}\omega_E.
\]
Note that since the form is harmonic, the line integrals do not depend on the path and we can write
\[
\int_{z_1}^{z_2}\int_{(x_1,y_1)}^{(x_2,y_2)}\omega_E = \int_{z_1}^{z_2}\int_{(x_1,y_1)}^{(x_1,\infty)}\omega_E - \int_{z_1}^{z_2}\int_{(x_2,y_2)}^{(x_2,\infty)}\omega_E.
\]
Therefore, only the differential form corresponding to $\frac{dy}{y}\wedge {d\bar z}$ gives a non-zero contribution. From the Fourier expansion of $\omega_E$ we see see that the Fourier expansion of this term is
\[
F_2 \frac{dy}{y}\wedge d \bar z = \sum_{\substack{\alpha \in \cO_F\\ \alpha_0 >0}} \frac{a_{(\alpha)}}{N_{F/\Q}(\delta)} e^{{-2\pi 
i}\left(\frac{\alpha_0 \bar z}{\delta_0} + \frac{\alpha_1 x}{\delta_1} + \frac{\alpha_2 \bar x}{\delta_2} \right)} 
yK_0\left(4\pi \frac{|\alpha_1|}{|\delta_1|} y \right){dy}\wedge d\bar z.
\]
 Using that 
\[
\int t K_0(t) dt = -t K_1(t)
\]
we find the formula
\[
\int_{z_1}^{z_2}\int_{(x_1,y_1)}^{(x_1,\infty)}\omega_E = \sum_{\substack{\alpha \in \cO_F\\ \alpha_0 >0}} \frac{a_{(\alpha)}}{N_{F/\Q}(\delta)} \frac{e^{-2\pi i \left( \frac{\alpha_1}{\delta_1}x_1+  \frac{\alpha_2}{\delta_2}x_2\right)}}{-8\pi^2 i \frac{\alpha_0|\alpha_1|}{\delta_0|\delta_1|}}  \left( e^{-2\pi i \frac{\alpha_0}{\delta_0}\bar z_2} -  e^{-2\pi i \frac{\alpha_0}{\delta_0}\bar z_1}\right)y_1 K_1\left( \frac{4\pi |\alpha_1|}{|\delta_1|}y_1\right).
\]
This can readily be used in computations, since the coefficients $a_{(\alpha)}$ may be obtained in practice by counting points in the reductions of $E$ modulo the primes dividing $(\alpha)$.

\begin{example}
 \label{example: explicit computations and examples-arch} We end this section by exhibiting an explicit calculation of an archimedean Darmon point over a cubic field of mixed signature. Let $F=\Q(r)$ where $r$ satisfies the polynomial
\[
r^{3} - r^{2} + 1.
\]
The cubic field $F$ has discriminant $-23$, and signature $(1,1)$. The different ideal of $F$ is generated by $\delta=-3 r^{2} - r + 2$. Consider the elliptic curve $E/F$ given by the equation:
\[E: y^2 + \left(r - 1\right)xy + \left(r^{2} - r\right)y = x^3 + \left(-r^{2} - 1\right)x^2 + r^{2}x.\]
The elliptic curve $E$ has prime conductor $\fN = \left(r^{2} + 4\right)$, of norm $89$.

In this example we will construct a point on $E$ attached to the quadratic extension $K=F(w)$, where $w$ satisfies the polynomial
\[
y^{2} + \left(r + 1\right) y + 2 r^{2} - 3 r + 3.
\]
It turns out that $w$ generates the ring of integers $\cO_K$ as a $\cO_F$-module. Moreover, the field $K$ has class number one, and thus we expect the point to be defined over $K$. 

Consider the (optimal) embedding $\psi\colon K\to M_2(F)$ of level $\fN$, given by:
\[
w\mapsto \left(\begin{array}{rr}
-2 r^{2} + 3 r & r - 3 \\
r^{2} + 4 & 2 r^{2} - 4 r - 1
\end{array}\right).
\]

Let $\gamma_\psi = \psi(u)$, where $u$ is a fundamental norm-one unit of $\cO_K$. In fact, $\gamma_\psi$ can be taken to be:
\[
\left(\begin{array}{rr}
-4 r - 3 & -r^{2} + 2 r + 3 \\
-2 r^{2} - 4 r - 3 & -r^{2} + 4 r + 2
\end{array}\right).
\]
It fixes the point
\[
\tau_\psi = -0.7181328459824\ldots + 0.55312763561813\sqrt{-1} \in \cHtwo.
\]

Summing over all ideals $(\alpha)$ of norm up to $400,000$, we obtain:
\[
J_\psi = 0.00052812842341311719013530664 + 0.0013607546066441620241871911551\sqrt{-1}.
\]

This yields:
\[
z_\psi = \frac{(2\pi i)^3 \sqrt{23}}{\Omega_E} J_\psi = 0.14196707701839569927696 - 0.055099463339094455920253\sqrt{-1},
\]
with
\[
\Omega_E = |\operatorname{Im}(\ol \Lambda_{1,0} \Lambda_{1,1})| = 11.402384864412804650783641396196847705711042.
\]
Here, $\Lambda_i = \Lambda_{i,0} + \sqrt{-1}\Lambda_{i,1}$ is the period lattice of $E/K$ with respect to the embedding $v_i$.

Finally, consider the point $P\in E(K)$ with coordinates
\[
P=\left(r - 1 : w - r^{2} + 2 r : 1\right),
\]
and let $z_P = -0.7203331961645979330006996731635\in \C/\Lambda_E$ be its preimage under the Weierstrass uniformization. The following equality holds, up to the $32$ digits of numerical accuracy to which we have computed $z_\psi$: 
\[
-100  z_\psi  - 10 z_P + 4\Lambda_{0,0} + 5\Lambda_{0,1} = 0.
\]
This gives numerical evidence of the fact that the Darmon point $z_\psi$ is, up to torsion points, a multiple of the point of infinite order  $P\in E(K)$.
\end{example}

\subsection{Non-archimedean computations}
\label{subsection: non-archimedean computations}

The first explicit methods for computing $p$-adic Darmon points associated to quaternion division algebras were introduced in \cite{shpquat}. Although in the setting of \cite{shpquat} the base field is $F=\Q$, $K$ is real quadratic, and $B/\Q$ is an indefinite division algebra, as we will see in this section the methods, in fact, can be easily adapted to work also in more general settings. That is to say, the algorithms of \cite{shpquat} can be suitably modified so that they also allow for the computation of some of the new instances of Darmon points that were introduced in Section \ref{section: non-archimedean Darmon points}.

Indeed, most of the methods developed in \cite{shpquat} are of a local nature. Namely, they just exploit the fact that $B\otimes\Q_p\simeq \M_2(\Q_p)$ and that \[\Gamma_0(pm)\otimes \Z_p\simeq \left\{\smtx a b c d \in\SL_2(\Z_p)\colon p\mid c \right\}.\]
The only place where the global nature of the quaternion algebra $B$ plays a key role is in the use of effective algorithms for computing generators and solving the word problem in $\Gamma_0(pm)$. In the setting of \cite{shpquat}, these are provided by John Voight's algorithms \cite{voight}. 

Therefore, the computational methods of \cite{shpquat} are also valid for other quaternion division algebras $B/F$ and arithmetic groups $\Gamma_0(\fp\fm)\subset B$, as long as
\begin{enumerate}
\item $B\otimes_F F_\fp\simeq \M_2(\Q_p)$ and $\Gamma_0(\fp\fm)\otimes \cO_{F,\fp}\simeq \left\{\smtx a b c d\in\SL_2(\Z_p)\colon p\mid c \right\}$; and
\item There are effective algorithms for computing generators and solving the word problem in $\Gamma_0(\fp\fm)$ .
\end{enumerate}

In this section we are interested in the case where the base field $F$ is of mixed signature. One can use then Aurel Page's algorithms \cite{page} for computing generators and solving the word problem. These algorithms are not completely general but can only be applied to a certain type of quaternion algebras. Thus, in order to use them we need to impose some restrictions on $F$ and $B$, which we describe next.

We assume from now on that  $F$ is a number field of signature $(r,1)$ and that $K$ is a quadratic totally complex extension of $F$. As usual, we consider an elliptic curve over $E/F$ and we assume that its conductor is squarefree and can be factored as 
$\fN = \fp\fd\fm$, where $\fp$ is inert in $K$ and $\fd$ (resp. $\fm$) denotes the product of primes that are inert (resp. split) in $K$. The sign of the functional equation of $L(E/K,s)$ is then given by the parity of 
\begin{align*}
  r+1+\#\{\fq \mid \fd\}.
\end{align*}
Thus, under the assumption that the sign is $-1$ we can, and will, take $B$ to be the quaternion algebra of discriminant $\fd$ that ramifies at all the real places. In other words, $B$ has discriminant $\fd$ and splits only at the complex place of $F$. The reason for this particular choice of algebra is that $B$ is then a so-called Kleinian quaternion algebra (see \cite[p.7]{page}). In addition, $\Gamma_0(\fm)\subset B$ is a Kleinian group and therefore we can use the algorithms of Aurel Page that compute presentations and solve the word problem in $\Gamma_0(\fm)$.

We also make the assumption that $F_\fp = \Q_p$, where $p=\fp\cap \Z$. This implies that $K_\fp = \Q_{p^2}$ (the quadratic unramified extension of $\Q_p$) and that $\cH_\fp =\cH_p = \Q_{p^2}\setminus \Q_p$.

Let $\psi\colon \cO_K\hookrightarrow R_0(\fm)$ be an optimal embedding. In the next paragraphs we include a brief description of the main steps that lead to an explicit computation of the Darmon point $J_{\psi}$. However, we refer to \cite{shpquat} for more details, as well as for the complete proofs which, as remarked in the discussion above, remain valid in the present context.

A technical, but crucial, step in \cite{shpquat} is the choice of a particular set of elements in $\Gamma= \left(R_0(\fm)\otimes_{\cO_F}\cO_{F,\{\fp\}}\right)^\times_1$ representing the edges of the Bruhat--Tits tree of $\PGL_2(\Q_p)$. In order that the algorithms go through in the current setting one needs also to make a careful choice of such representatives, as we describe next. (This is essentially the only modification from \cite{shpquat} that needs to be done.) Let $\Upsilon = \{\gamma_0,\dots,\gamma_p\}$ be a system of representatives of $\Gamma_0(\fp\fm)\backslash\Gamma$, where the $\gamma_i$'s are chosen to be locally of the form
\begin{align*}
  \gamma_0 = 1, \ \iota_\fp(\gamma_i)= u_i\smtx{0}{-1}{1}{i} \text{ for } i =1,\dots,p,
\end{align*}
for some $u_i$ belonging to
\begin{align*}
  \Gamma_0^{\text{loc}}(p)= \left\{ \smtx{a}{b}{c}{d} \in \SL_2(\Z_p) \colon  p\mid c \right\}.
\end{align*}
In addition, we choose $\omega_\fp\in R_0(\fp\fm)$ an element of norm $\fp$ that normalizes  $\Gamma_0(\fm)$ and is locally of the form 
\begin{align*}
  \iota_{\fp}(\omega_\fp) = u\cdot \smtx{0}{-1}{\pi}{0},
\end{align*}
where $\pi$ is a generator of $\fp$ and $u\in\Gamma_0^{\text{loc}}(p)$. With this choice of $\omega_\fp$ we let $\widehat\Gamma_0(\fm)= \omega_\fp\Gamma_0(\fm)\omega_\fp^{-1}$, and we take as representatives for $\Gamma_0(\fp\fm)\backslash\widehat\Gamma_0(\fm)$ the elements $\tilde\gamma_i$ defined by: 
\begin{align*}
\tilde\gamma_i =\begin{cases}
1&i=0,\\
\pi^{-1}\omega_\fp\gamma_i\omega_\fp&i=1,\ldots,p.
\end{cases}
\end{align*}
Let $\cE^+$ denote the even edges of the Bruhat--Tits tree of $\PGL_2(\Q_p)$. Then $\Gamma$ acts transitively on $\cE^+$, and if we let $e_*$ denote the principal edge the map $g\mapsto g^{-1}(e_*)$ induces a bijection $\cE^+\simeq \Gamma_0(\fp\fm)\backslash \Gamma$. Since $\Gamma = \Gamma_0(\fm)\star_{\Gamma_0(\fp\fm)}\widehat\Gamma_0(\fm)$ we see that the representatives $\{\gamma_i\}_i$ for $\Gamma_0(\fm)$ and $\{\tilde\gamma_i\}$ for $\widehat\Gamma_0(\fm)$ chosen above determine a set of representatives $\mathcal{Y}$ for $\Gamma_0(\fp\fm)\backslash\Gamma$. If $e\in\cE^+$ we denote by $\gamma_e$ the element in $\mathcal{Y}$ such that $\gamma_e^{-1}(e_*)=e$. Then the elements in  $\mathcal{Y} $ can be labeled as $\{\gamma_e\}_{e\in\cE^+}$ and $\mathcal{Y} $ is a system of representatives for $\Gamma_0(\fp\fm)\backslash\Gamma$ which is a so-called \emph{radial system} (cf. \cite[Definition 4.7]{LRV}).

\subsubsection{The homology class}
Let $\tau_{\psi}\in\cH_\fp$ be an element fixed by the action of $\iota_\fp\circ \psi(K)$. Let also $\varepsilon\in \cO_K^\times$ be a unit such that $\Nm_{K/F}(\varepsilon)=1$, and set $\gamma_{\psi} = \psi(\varepsilon)$. The element $\gamma_{\psi}\otimes\tau_{\psi}$ is a $1$-cycle and it defines a homology class in $H_1(\Gamma_0(\fm),\Div\cH_\fp)$. There exists $e\in\Z_{>0}$ such that $\gamma_{\psi}^e\otimes\tau_{\psi}$ is homologous to a cycle in $Z_1(\Gamma,\Div^0\cH_\fp)$, which gives precisely the homology class attached to $\psi$. The method described in \cite[\S 4]{shpquat} is valid in the present setting, and it can be used as an effective algorithm for explicitly computing a cycle $\Delta_{\psi}$ that is homologous to $\gamma_{\psi}^e\otimes \tau_{\psi}$ and is of the form
\begin{align}
  \Delta_{\psi} = \sum_i g_i\otimes(x_i- y_i), \text{ for some } g_i\in\Gamma,x_i,y_i\in\cH_\fp.
\end{align}

\subsubsection{The cohomology class} The first step is to compute an element \[\varphi_E\in \left( H^1(\Gamma_0(\fp\fm),\Z)^{\fp-\text{new}}\right)^{\lambda_E}.\] Namely, a cohomology class in $H^1(\Gamma_0(\fp\fm),\Z)$ that is $\fp$-new and lies in the isotypical component determined by the character  $\lambda_E$. One can make this step effective by noting that
\begin{align*}
  H^1(\Gamma_0(\fp\fm),\Z)\simeq \Gamma_0(\fp\fm)_{\text{ab}}.
\end{align*}
Using again the algorithms provided by \cite{page} one can compute generators and relations for $\Gamma_0(\fp\fm)$, hence also for $\Gamma_0(\fp\fm)_{\text{ab}}$. It is easy to compute the Hecke action on these generators, and obtain the matrices for the different operators $T_\mathfrak{l}$. Then $\varphi_E$ is found by taking the common eigenspace for several of these $T_\mathfrak{l}$.

Combining Shapiro's Lemma with the isomorphism $\cF(\cE^+,\Z)\simeq \Gamma_0(\fp\fm)\backslash \Gamma\simeq \text{Ind}_{\Gamma_0(\fp\fm)}^\Gamma$ we obtain that 
\begin{align*}
  H^1(\Gamma_0(\fp\fm),\Z)\simeq H^1(\Gamma,\cF(\cE^+,\Z)).
\end{align*}
Let $\tilde\mu_E$ denote the image of $\varphi_E$ under the above isomorphism. In fact, in order to lighten the notation let us set $\tilde\mu = \tilde\mu_E$ and $\varphi = \varphi_E$.  Since the isomorphism in Shapiro's Lemma is explicit, we have the following formula for $\tilde\mu$:
\begin{align*}
  \tilde\mu_g(e) = \varphi_{h(g,e)},\ \text{ for } g\in\Gamma,e\in\cE^+.
\end{align*}
Here $h(g,e)\in\Gamma_0(\fp\fm)$ is the element determined by the identity
\begin{align*}
  \gamma_e g = h(g,e)\gamma_{g^{-1}(e)},
\end{align*}
and it can be algorithmically computed by means of \cite[Theorem 4.1]{shpquat}.

Finally, let $\mathfrak{l}$ be an ideal prime to $\fp\fm$, and consider the cocycle  $(T_{\mathfrak{l}}-|\mathfrak{l}|-1)  \tilde\mu_E$, whose cohomology class belongs to $ H^1(\Gamma,\cF(\cE^+,\Z))$. It turns out (cf. \cite[Proposition 3.4]{shpquat}) that $(T_{\mathfrak{l}}-|\mathfrak{l}|-1)  \tilde\mu_E$ lies in the image of the natural map
\begin{align*}
  H^1(\Gamma,\text{HC}(\Z))\ra H^1(\Gamma,\cF(\cE^+,\Z)),
\end{align*}
and that it is actually a multiple of the cohomology class associated to $E$. Since  $T_{\mathfrak l}$ can be explicitly calculated, we see that we can effectively compute also the cohomology class $\mu_E$ associated to $E$.

\subsubsection{The integration pairing} In the previous paragraphs we have seen how to compute the cohomology class  $\mu = \mu_E$ associated to $E$, as well as the homology class $\Delta_{\psi}$ associated to the optimal embedding $\psi\colon \cO_K\hookrightarrow R_0(\fm)$.  Therefore, in order to compute the Darmon point $P_\psi$ we need to compute integrals of the form
\begin{align*}
I=  \Xint\times_{\P^1(\Q_p)} \left(\frac{t-\tau_2}{t-\tau_1} \right)d\mu_{g} \ \text{for }g\in \Gamma\text{ and } \tau_1,\tau_2\in\cH_\fp.
\end{align*}
Since we have computed $\mu$ explicitly, it would be possible to compute this integrals by means of Riemann products. However, this turns out to be too inefficient, and it is better to use an alternative method based on overconvergent cohomology. In the present setting, the algorithm of \cite[\S 5]{shpquat} can be used without any essential modification.

\subsubsection{Explicit computations and examples} \label{subsubsection: explicit computations and examples}

Finally, we end this section by presenting two explicit numerical calculations of Darmon points, computed using the above methods.

\begin{example}\label{ex: non-arch 1}\textbf{Case of cubic base field of mixed signature ($r=1$, $n=0$, $s=1$).}

Let $F = \QQ(t)$ with $t$ a root of the polynomial $m_t(x) = x^3-x^2-x+2$. This field has discriminant $\Delta_F=-59$. Consider the elliptic curve $E/F$ given by the equation
\[
y^2 + \left(-t - 1\right) x y + \left(-t - 1\right) y = x^{3} - t x^{2} + \left(-t - 1\right) x.
\]
 Its conductor factors as $\fN = \left(t^{2} + 2\right) =\fp\fq$, where
\[
\fp = \left(-t^{2} + 2 t + 1\right),\quad \fq = \left(t\right).
\]
 The prime $\fp$ has norm $17$, while $\fq$ has norm $2$.
 We consider the quaternion algebra $B/F$ ramified precisely at $\fq$ and at the real place of $F$. It has the presentation
\[
B = F\langle i,j,k\rangle,\quad i^2=-1, j^2=r,ij=-ji=k.
\]
 Consider $K=F(\alpha)$, where $\alpha^2 = -3t^2 + 9t - 6$.
 The maximal order of $K$ is generated by $w_K$, a root of the polynomial
\[
x^2+(t+1)x+7/16t^2-1/16t+5/8.
\]
 One can embed $\cO_K$ into an Eichler order of level $\fp$ by sending $w_K$ to the element
\[
(-t^2 + t)i + (-t + 2)j + tk.
\]
 We take $\gamma_\psi$ to be the element
\begin{align*}
\gamma_\psi &= 3t^2 - 7/2 + (t + 3/2)i + (t^2 + 3/2t)j + (5/2t^2 - 7/2)k
\end{align*}
with fixed point
\begin{align*}
\tau_\psi &= (12g + 8) + (7g + 13)\cdot 17 + (12g + 10)\cdot 17^2 + (2g + 9)\cdot 17^3 + (4g + 2)\cdot 17^4 + \cdots 
\end{align*}
where $g\in\cH_p$ satisfies the same polynomial as the one satisfied by $w_K$.

 Computing the integrals up to a precision of $60$ $17$-adic digits we obtain
\[
J_\psi = 16 + 9\cdot 17 + 15\cdot 17^2 + 16\cdot 17^3 + 12\cdot 17^4 + 2\cdot 17^5 + 13\cdot 17^6 + 12\cdot 17^7 + 2\cdot 17^8 + \cdots 
\]
which coincides, up to the working precision, with the following global point of infinite order:
\[
P_\psi = -\frac{3}{2}\cdot 72\cdot \left(t-1, \frac{\alpha + t^2 + t}{2}\right)\in E(K).
\]

\end{example}

\begin{example}\label{ex: non-arch 2}\textbf{Case of quadratic imaginary base field ($r=0$, $n=0$, $s=1$).}

Let  $F = \QQ(\sqrt{-2})$ and let $E/F$ be the elliptic  curve given by the equation
\[
y^2 + \left(\sqrt{-2} + 1\right) x y + \left(\sqrt{-2} + 1\right) y = x^{3} + \left(-\sqrt{-2} + 1\right) x^{2} + \left(-2 \sqrt{-2}\right) x - \sqrt{-2}.
\]
Its conductor is $\fN = \left(-3 \sqrt{-2} - 9\right) =\fp\fp'\fq$, where
\[
\fp = \left(\sqrt{-2} + 1\right),\quad \fp' = \left(-\sqrt{-2} + 1\right),\quad \fq = \left(-\sqrt{-2} - 3\right).
\]
 Both $\fp$ and $\fp'$ have norm $3$, while $\fq$ has norm $11$. We consider the quaternion algebra $B/F$ ramified precisely at $\fp'$ and $\fq$, which is given by
\[
B = F\langle i,j,k\rangle ,\quad i^2=-1,\ j^2=2\sqrt{-2}-5,\ ij=-ji=k.
\]
We take the quadratic extension $K=F(\sqrt{2})=\Q(\sqrt{-1},\sqrt{2})$. The ring of integers $\cO_K$ is generated over $\cO_F$ by $w_K$, a root of the polynomial $x^2-\sqrt{-2}x-1$.
 The element $\gamma_\psi = 3 + (-8\sqrt{-2} - 8)i + 4j + (2\sqrt{-2} - 4)k$ has reduced norm $1$ and lies in the image of a certain optimal embedding $\psi$. The corresponding fixed point is
\[
\tau_\psi = (2g + 2) + (g + 1)\cdot 3 + 2\cdot 3^2 + 2g\cdot 3^3 + 2g\cdot 3^4 + 2\cdot 3^5 + 2g\cdot 3^6 + (g + 1)\cdot 3^7 + \cdots 
\]
where $g\in\cH_\fp$ satisfies $x^2-\sqrt{-2}x-1=0$. We computed the integrals with a a precision of $60$ $3$-adic digits, and the obtained Darmon point is
\[
J_{\psi} = 1 + 3 + 3^2 + 2\cdot 3^3 + 3^4 + 2\cdot 3^8 + 2\cdot 3^9 + 3^{10} + 3^{13} + \cdots
\]
This coincides, up to the working precision of $60$ $3$-adic digits, with the image under Tate's uniformization map of the following global point of infinite order: 
\[
P_\psi =  -4 \cdot 16 \cdot \left(-\frac{\sqrt{-2}+5}{9} , \frac{4\sqrt{-2}-7}{54}\alpha - \frac{\sqrt{-2}+2}{6} \right)\in E(K).
\]
\end{example}

\bibliographystyle{halpha}
\bibliography{refs}

\end{document}